\documentclass[10pt]{amsart} 
\usepackage[latin1]{inputenc}
\usepackage[T1]{fontenc} 
\usepackage{url}
\usepackage[all]{xy}
\usepackage{textcomp}

\usepackage{amsfonts}
\usepackage{amssymb}
\usepackage{amsthm}

\usepackage{stmaryrd} 

\usepackage{graphicx}

\usepackage{mathrsfs}
\usepackage{bbm}
\usepackage{oldgerm}

\usepackage{hyperref} 
\hypersetup{hidelinks}
\usepackage[alphabetic]{amsrefs}

\usepackage[shortlabels]{enumitem} 

\usepackage{epigraph}

\usepackage{tikz}

\usepackage{mathtools} 

\newcommand{\iso}{\cong}            

\newcommand{\Hom}[3]{\mathrm{Hom}_{#1}\bigl(#2,#3\bigr)}   
\newcommand{\End}[2]{\mathrm{End}_{#1}\bigl(#2\bigr)}    
\newcommand{\Gal}[2]{\mathrm{Gal}({#1}/{#2})}   

\DeclareMathOperator{\Frac}{Frac}

\DeclareMathOperator{\Spec}{Spec}

\DeclareMathOperator{\CHM}{\mathcal{M} ot}
\DeclareMathOperator{\NUM}{NUM}
\DeclareMathOperator{\Sym}{Sym}

\newcommand{\ZZ}{\ensuremath{\mathbb{Z}}} 
\newcommand{\QQ}{\ensuremath{\mathbb{Q}}} 
\newcommand{\RR}{\ensuremath{\mathbb{R}}} 
\newcommand{\CC}{\ensuremath{\mathbb{C}}} 
\newcommand{\Fp}{\ensuremath{\mathbb{F}_{p}}}

\newcommand{\Qp}{\ensuremath{\mathbb{Q}_p}}    
\newcommand{\Zp}{\ensuremath{\mathbb{Z}_p}}    
\newcommand{\Cp}{\ensuremath{\mathbb{C}_p}} %
\newcommand{\Bdr}{\ensuremath{\mathrm{B}_{\mathrm{dR}} }}      
\newcommand{\Bcris}{\ensuremath{\mathrm{B}_{\mathrm{cris}} }}  

\newcommand{\FiGa}[1][]
{\ensuremath{{\Phi\!\Gamma}_{{#1}}^{\mathrm{\acute{e}t}}}}

\newcommand{\mca}[1]{\ensuremath{\mathcal{#1}}}

\DeclareFontFamily{T1}{rsfs}{}
\DeclareFontShape{T1}{rsfs}{m}{n}{ <-> rsfs10 }{}

\newcommand{\id}[1][]{\ensuremath{\mathrm{Id}_{#1} }}

\DeclareMathOperator{\tr}{tr}

\renewcommand{\det}{\operatorname{det}} 

\newcommand{\Fro}[1][]{\ensuremath{\varphi_{#1}} }

\newcommand{\cris}{_\mathrm{cris}}
\newcommand{\dR}{_\mathrm{dR}}

\newcommand{\eps}{\ensuremath{\varepsilon} }
\newcommand{\Fil}[1][i]{\ensuremath{\mathrm{Fil}^{#1}} }

\newcommand{\nr}{\ensuremath{^\mathrm{nr}}}

\newcommand{\num}{\mathrm{num}}
\newcommand{\prim}{\mathrm{prim}}

\newcommand{\one}%
{{\rm 1\hspace*{-0.4ex}\rule{0.1ex}{1.52ex}\hspace*{0.2ex}}}

\newcommand{\Fa}{\ensuremath{F_a}} 
\newcommand{\Fatilde}{\tilde{F}_a}
\newcommand{\La}{\ensuremath{L_a}}

\newcommand{\cf}{z} 

\newcommand{\Qpbar}{\overline{\mathbb{Q}}_p}

\newcommand{\Rcris}{\ensuremath{R\cris}}
\newcommand{\RdR}{\ensuremath{R\dR}}

\newcommand{\BcrisF}{\ensuremath{\mathrm{B}_{\mathrm{cris},F} }}  
\newcommand{\BcrisL}{\ensuremath{\mathrm{B}_{\mathrm{cris},L} }}

\newcommand{\BcrisFtilde}{\ensuremath{\mathrm{B}_{\mathrm{cris},\tilde{F}}} }
\newcommand{\vdR}{v\dR}

\newcommand{\Vcrisv}{V\cris^{\vee}}

{\begin{list}{\roman{enumi})}{\usecounter{enumi}} }%
{\end{list}}

\newenvironment{meta}{%
\sffamily[}{\upshape]}

\newenvironment{dem}[1][Proof.]
{\begin{proof}[#1]}{\end{proof}}

\newtheorem{theorem}[subsection]{Theorem} 

\newtheorem{prop}[subsection]{Proposition}
\newtheorem{lemma}[subsection]{Lemma}
\newtheorem{cor}[subsection]{Corollary}

\newtheorem*{claim*}{Claim}
\newtheorem{conj}[subsection]{Conjecture}

\theoremstyle{definition}
\newtheorem{definition}[subsection]{Definition}

\newtheorem*{definition*}{Definition}

\newtheorem{conv}[subsection]{Convention}

\theoremstyle{remark}
\newtheorem{remark}[subsection]{Remark}
\newtheorem{remark*}{Remark}

\numberwithin{equation}{subsection}   

\title{The Hilbert symbol in the Hodge standard conjecture} 

\begin{document}

\author{Giuseppe Ancona \and Adriano Marmora}
\address{IRMA, Université de Strasbourg et CNRS, 
7 rue René-Descartes,
67084 Strasbourg, France
}
\email{ancona@math.unistra.fr, marmora@math.unistra.fr}

\thanks{ This research was partly supported by the grant ANR--23--CE40--0011 of \emph{Agence National de la Recherche}.
	}

\date{\today}

\begin{abstract}
We study the Hodge standard conjecture for varieties over finite fields admitting a CM lifting, such as abelian varieties or products of K3 surfaces. For those varieties we show that the signature predicted by the conjecture  holds true modulo $4$. This amounts to determining the discriminant and the Hilbert symbol of the intersection product. The first is obtained by $\ell$-adic arguments whereas the second needs a careful computation  in $p$-adic Hodge theory.
  \end{abstract}

\maketitle

\setcounter{tocdepth}{1} 
\tableofcontents


\section{Introduction} 

The standard conjecture of Hodge type predicts the signature of the intersection product of algebraic classes on a smooth projective variety. In this paper we study the discriminant and the Hilbert symbol of the intersection product and we show that 
they coincide with those predicted by the conjecture for varieties over finite fields which admit a CM lifting, in particular for abelian varieties and for  products of K3 surfaces.  
This can be reformulated by saying that the expected signature holds true modulo $4$. 

This conjecture was formulated by Grothendieck in the sixties \cite{Gro}. It is modelled  on positivity results such as the Hodge index theorem, the Hodge--Riemann bilinear relations and the positivity of the Rosati involution. For a panorama on the history, the original motivation and  potential applications of this conjecture see \cites{MR2115000, MR4199442, IIK}. 
This conjecture did not progress until the 21st century. Then, Milne showed that the classical Hodge conjecture for complex abelian varieties would imply the Hodge standard conjecture for abelian varieties in positive characteristic \cite{MR1906596} and Ito studied the behaviour of this conjecture under blow-ups \cite{MR2125735}. 

Very recently, some progress has been made. The first author proved the conjecture for  motives of rank two using $p$-adic Hodge theory  \cite[Theorem 8.1]{MR4199442}. This result implies the conjecture for abelian fourfolds \cite[Theorem 1.3]{MR4199442} and for some other abelian varieties \cite{Koshi_HSC_prime,Agug}. On the other hand, Ito--Ito--Koshikawa proved the conjecture for the square of a K3 surface, using the Kuga--Satake construction and ultimately  relying on the positivity of the Rosati involution \cite{IIK}.

The purpose of this article is to push the $p$-adic methods initiated in \cite{MR4199442} to CM motives of higher rank
and  to describe completely the Hilbert symbol 
of the intersection product.
To be more precise, let us recall the formulation of the standard conjecture of Hodge type. 
 \begin{definition}\label{def intro}
Let $X$ be  a smooth, projective and geometrically connected variety of dimension $d$ and let $L$ be a hyperplane section  of $X$. We denote by $\mathcal{Z}^{n}_{\num}(X)_{\QQ}$ the finite dimensional $\QQ$-vector space of $\QQ$-algebraic cycles on $X$ of codimension $n$ modulo numerical equivalence.
     For $n \leq d/2$ we define the space of primitive cycles  as
\[\mathcal{Z}^{n,\prim}_{\num}(X,L)_{\QQ} \coloneqq \{\alpha \in \mathcal{Z}^{n}_{\num}(X)_{\QQ} , \hspace{0.2cm} \alpha \cdot L^ {{d-2n+1}}=0 \hspace{0.2cm}  \textrm{in}  \hspace{0.2cm}  \mathcal{Z}^{d-n+1}_{\num}(X)_{\QQ}\}\]
and we define the pairing  
$q_{X,L,n}: \mathcal{Z}^{n,\prim}_{\num}(X,L)_{\QQ}  \times  \mathcal{Z}^{n,\prim}_{\num}(X,L)_{\QQ}  \rightarrow \QQ$
 via the intersection product
 \[\alpha, \beta \mapsto  (-1)^n \alpha \cdot \beta \cdot L^{d-2n}.\]
 \end{definition}

   \begin{conj}{(Hodge standard conjecture, \cite[Conjecture 2]{Gro}.)} Let $X$ be  a smooth, projective and geometrically connected variety of dimension $d$ and let $L$ be a hyperplane section  of $X$. Then for all $n \leq d/2$, the quadratic form $q_{X,L,n}$ 
is positive definite. 
%
\end{conj}
 \begin{remark}
	The original formulation of \cite{Gro} is with cycles modulo homological equivalence. As another standard conjecture predicts that homological and numerical equivalence should coincide, the two formulations should be equivalent. See also \cite[3.11-3.12]{MR4199442}.
\end{remark}
Our main result is the following.
\begin{theorem}\label{thmintro}
\begin{enumerate}
\item
Let $(X,L)$ be an abelian variety over a finite field  together with a hyperplane section induced by a CM lifting. Let $n \leq \dim(X)/2$ be any integer and  $(s_+,s_-)$ be the signature of $q_{X,L,n}$. Then $s_-$ is divisible by $4$.
\item
Let $X=S_1\times   \ldots \times S_m$ be the product of  K3 surfaces over a finite field, $L_i$ be a hyperplane section induced by a  CM lifting of $S_i$, {\itshape cf.} 
\cite{MR4241794},    and let $L=\oplus_i p^*_i L_i$ be the induced hyperplane section on $X$. Let $n \leq \dim(X)/2=m$ be any integer and   $(s_+,s_-)$ be the signature of $q_{X,L,n}$. Then $s_-$ is divisible by $4$.
\end{enumerate}
\end{theorem}
The above theorem follows from   a general result on motives endowed with quadratic forms (see Theorem \ref{mainthm} for the precise definitions).
\begin{theorem}\label{thmintromot}
Let $M$ be a CM motive  in mixed characteristic with coefficients in $\QQ$. Suppose that its special fiber $M_p$ is supersingular and that $M$ is endowed with a CM quadratic form $q$ whose Betti realization is a polarisation of the underlying Hodge structure. Let $q_Z$ be the restriction of $q$ to all algebraic classes of $M_p$. Write $(s_+,s_-)$ for the signature of $q_Z$, then  $s_-$ is divisible by $4$.\end{theorem}

Let us give a sketch of the proof of Theorem \ref{thmintromot}.
For a quadratic form over $\QQ$, the condition $4\vert s_-$ is equivalent to the fact that the discriminant is positive and that the Hilbert symbol at infinity is $+1$. Both conditions can be checked by studying the quadratic form in the non-archimedean   places, thanks to the classical product formula on Hilbert symbols. In particular,  Theorem \ref{thmintromot} boils down to the study of the quadratic form $q_Z$ after its $\ell$-adic and crystalline realization. The $\ell$-adic computation is obtained  from general results on $\ell$-adic cohomology, and it is sufficient to prove that the discriminant is positive, as in \cite{MR4199442}.
Instead, to compute the Hilbert symbol, we need to control the crystalline contribution.  Our arguments  rely heavily on $p$-adic Hodge theory. 
Let us give the mainlines below.

We begin by extending the coefficients of  the motive $M$ from $\QQ$ to $\Qp$. Precisely, let $F$ be the CM number field acting on $M$, the motive $M\otimes_{\QQ} \Qp$ is endowed with an action of $F \otimes_{\QQ} \Qp$. Such an action decomposes the motive into a sum of direct factors and we study them separately. Thus, we reduce  to the situation where the motive $M$ has an action of a finite extension of $\Qp$, which we will denote again by $F$, 
endowed with a non trivial involution $*\colon F\rightarrow F$. The crystalline and $p$-adic étale realisations of $M$ are endowed by hypothesis with a non-degenerate quadratic form. 
Faltings's crystalline comparison  theorem \cite{MR1463696} allows to compare these realisations via a matrix of periods laying in Fontaine's ring $\Bcris$ of $p$-adic periods.  By using the CM action we characterise this matrix by  a single invertible period $\lambda \in\BcrisF^{\times}$, where $\BcrisF$ is the smallest 
subring of $\Bdr$ containing $\Bcris$ and $F$. This period is unique up to a scalar in $F^{\times}$ and it controls completely the arithmetic of our problem. Namely it produces a ``re-normalisation factor'' $\lambda\lambda^*$ belonging to $F_0^{\times}$, where $F_0$ denotes the subfield of $F$ of elements fixed by $*$. Thanks to a theorem of Milnor  on CM quadratic forms on $p$-adic fields \cite{MR249519}, we show that our problem is equivalent to control precisely whether the re-normalisation factor is a norm  of an element of $F^{\times}$. We verify that this problem is multiplicative with respect to the tensor product of motives: hence, we are reduced to treat the case of small Hodge weights. For those motives we need to prove that the re-normalisation  factor $\lambda\lambda^*$ is not a norm.

This is the crucial point of the article.
  Even if the definition of the period $\lambda$, based on the crystalline comparison theorem, is not explicit,
	this period satisfies a strong uniqueness property with respect to the behaviour under the Frobenius and the de Rham filtration.
	Such a uniqueness follows from the fundamental exact sequences in $p$-adic Hodge theory.
		Therefore, we can describe $\lambda$ through any other periods with similar behaviour. For this purpose, we use 
	 	   Lubin--Tate periods of Colmez \cite{MR1956055} to get some control of $\lambda$.	  
		It turns out that such a control is precisely what we need to apply  
		 a theorem of Dwork 
\cite{MR98079} in order to compute the image of 
		the re-normalisation factor $\lambda\lambda^*$ through the local reciprocity map $F^{\times}_0\rightarrow \Gal{F}{F_0}$. 
As the kernel of the reciprocity map is the group of norms $N_{F/F_0}(F^{\times})$, we obtain that $\lambda\lambda^*$ is not  a norm.

\begin{remark}
The main result of \cite{MR4199442} is an instance of Theorem \ref{thmintromot}, namely the case where the motive $M$ has rank two. In that particular case the field $F$ is   a quadratic extension of $\Qp$.
As there  are finitely many such extensions, it was possible in \cite{MR4199442}*{Section 12} to compute the period $\lambda$ case by case.
Moreover, those periods could be described in terms of Lubin--Tate periods and elements algebraic over $F$.

In higher rank, $F$ is a quadratic extension of $F_0$,
whose degree over $\Qp$ is equal to the half of the rank of the motive,  with 
{\itshape a priori} no restriction on the ramification of $F_0/\Qp$ nor of $F/F_0$.
Hence an analysis case by case is impossible.
Moreover, the period $\lambda$ is computed in terms of Lubin--Tate periods and a non-explicit element in $\widehat{F}\nr$,  the completion of the maximal unramified extension of $F$. In general such an element is  
transcendental over $F$, see Remark \ref{rk:transcendente}.
\end{remark}

\subsection*{Organisation of the paper} 
We begin by setting up notation and conventions in Section \ref{sect:Conv}.
In Section \ref{sect:MS} we give the precise statement of Theorem \ref{thmintromot} and explain how to deduce Theorem \ref{thmintro} from it. In Section \ref{sect:Arch2padic} we recall classical results on quadratic forms and use them, combined with general theorems on $\ell$-adic cohomology, to translate Theorem \ref{thmintromot} into a $p$-adic question. Section~\ref{sect:SObj} presents the first reduction step. We extend the coefficients of the motive from $\QQ$ to $\Qp$ and decomposes it  into a sum of direct factors. The goal of the section is to show that it is enough to study each factor separately.

In the rest of the paper we fix one of these factors. It is endowed with the action of a finite extension $F$ of $\Qp$. Section \ref{sect:CQ-qf} recalls classical results on CM quadratic forms over the $p$-adic numbers. At the end of the section we define the re-normalisation factor.
 This is used in Section~\ref{sect:Rtg} where we study the behaviour of the problem under tensor product. This allows to reduce the problem to motives whose crystalline realization has small Hodge weights. Section \ref{sect:p-adic_periods} gives a characterization  of the $p$-adic periods of those motives. This allows in Section \ref{sect:LT_and_reciprocity} to describe them explicitly by using Lubin--Tate periods and to apply Dwork's theorem  computing the local reciprocity.

In Section \ref{sect:proof_MT} we put all the ingredients together and give the proof of the main result. This section can also be read first in order to have a global picture of the strategy.

We finish with an appendix concerning the filtered $\Fro$-modules attached to a Lubin--Tate character of a $p$-adic field. 
We added it to keep the paper as self-contained as possible, by giving an alternative  construction of Lubin--Tate periods of \cite{MR1956055}.

\subsection*{Acknowledgements} We thank Eva Bayer for her explanations on CM quadratic forms and pointing out the reference \cite{MR249519}.  We thank Pierre Colmez, Lionel Fourquaux and Stefano Morra for discussions on Lubin--Tate formal groups and their filtered $\Fro$-modules.
We thank Emiliano Ambrosi and Daniel Kriz for useful comments on  preliminary versions of the text.
We thank the referees for their careful readings and relevant remarks.

 Some years ago we had some illuminating conversations on $p$-adic periods with our colleague Jean-Pierre Wintenberger. We dedicate this work to his memory.


\section{Conventions} \label{sect:Conv}

Throughout the paper we will work with the following notation and conventions.


\subsection{Involutions}\label{choucroute} We  usually denote by $*$ an involution acting on a set.   By abuse of notation we will still write  $*$ for an endomorphism of a ring $B$ extending an involution of a subfield $F\subset B$. We will denote by $F_0$ the subfield of $F$ where $*$ acts as the identity.

We will often write $z^*$ for the image $*(z)$ of an element $z$ through $*$.

\subsection{$p$-adic fields.} \label{sb:Conv_p-adic_fields} 
Let $p$ be a prime number. In this text,
by $p$-adic field we mean a finite extension of the field of $p$-adic numbers $\Qp$. For such a field $L$, we denote by $\mca{O}_L$ its ring of integers and by  $k_L$ the residue field. The degree of $k_L$ over $\Fp$ is called the residual degree of $L/\Qp$ and it will be denoted by $f_L$. We denote\footnote{In literature $\La$ is often denoted with subscript $0$ instead of $a$, but the former is already used in the context of involutions, as Convention \ref{choucroute}. 
	The subscript ``$a$'' stands for ``absolute unramified'' as $\La$ is also called the absolute unramified subextension of $L$.} 
by $\La$ the maximal subfield of $L$ unramified over $\Qp$. It is equal to the field of fractions of the ring of Witt vectors $W(k_L)$. The degree $e_L\coloneqq [L:\La]$ is called the ramification index of $L/\Qp$, or the \emph{absolute} ramification index of $L$.   
When there is no risk of confusion, we may drop the subscript $L$ in the notation.

We choose an algebraic closure $\Qpbar$ of $\Qp$ and we denote by $\overline{\mathbb{F}}_p$ its residue field. 
We put $\widehat{\QQ}_p\nr\coloneqq \Frac W(\overline{\mathbb{F}}_p)$. It is a completion of the maximal unramified extension $\Qp\nr$ of $\Qp$ in $\Qpbar$. The fields $\widehat{\QQ}_p\nr$ and $\Qp\nr$ are endowed with an automorphism $\Fro$, called the absolute Frobenius, which is the unique map lifting the $p$-power map on the residue field $\overline{\mathbb{F}}_p$.
For any subextension $L$ of $\Qpbar/\Qp$ we put $G_{L}\coloneqq \Gal{\Qpbar}{L}$.

\subsection{Motives.} \label{sb:Conv_Motives}
We will work with the category $\CHM(S)_{\QQ}$ of homological motives over a base $S$ with coefficients in $\QQ$.  
This category is defined as the quotient (in the sense of \cite[\S2.3]{And05})  of the category of Chow motives over $S$, {\itshape cf.}~\cite[\S5.1]{MR2795752}, with respect to the homological equivalence. 

In our article the base $S$ will be $\CC$, a finite field or the ring of integers of a $p$-adic field.
For generalities on $\CHM(S)_{\QQ}$ when $S$ is a field, we refer to  \cite{MR2115000}*{\S4}. 
When $S$ is the ring of integers of a $p$-adic field, see also the conventions in \cite{AncFra}. 
In general these categories depend on the chosen Weil cohomology, but we will use only cohomologies for which the classical comparison theorems will ensure that those are independent of the choice.

We will also use the quotient category of motives with respect to numerical equivalence and denote it by $\NUM(S)_\QQ$. We will use it only with $S$ being a finite field.
%

\subsection{Realisations} \label{sb:Conv_Realis}
We will make use of the classical realization functors, namely  the de Rham realization $\RdR$, the Betti realization $R_{B}$, the $\ell$-adic realization $R_{\ell}$ and the crystalline realization $\Rcris.$ We consider them with  their enriched structures, as in \cite[\S 7.1]{MR2115000}. In particular, the functor  $\RdR$ will land in the category of filtered vector spaces,   $R_{B}$  in the category of Hodge structures, $R_{\ell} $   in the category of Galois representations and  $\Rcris$ in the category of modules endowed with an action of an absolute Frobenius.

For a motive $M$, we will denote by $\dim M$ the dimension of any classical realization of $M$ and we call it the dimension of the motive.

\subsection{Unit object.} The unit object in any tensor category we will consider will be denoted by $\one$.

\subsection{Filtered Modules} \label{sb:Conv_Fil}
Let $M$ be a module over some ring and $(\Fil M)_{i\in\ZZ}$ a decreasing, exhaustive and separated  filtration by submodules. 
We will use the following conventions: for any subset $S\subset M$, we put  
$$\vdR(S)\coloneqq \sup\{i\in\ZZ|\, S \subset \Fil[i]M\}.$$
For $m$ in $M$, we write $\vdR(m)\coloneqq\vdR(\{m\})$.
(The subscript $\mathrm{dR}$ comes from the fact that the filtrations in this article will  be the de Rham filtration on some cohomology group or on some ring of periods.)


\subsection{$p$-adic Hodge theory.}  \label{sb:Conv_pHT}
We denote  Fontaine's rings of periods (associated with $\Qpbar/\Qp$) by \Bdr\ and \Bcris.
They were introduced in \cites{MR657238,MR1293971}. 
We gather below some properties of these rings that we will use in this article. 
For a detailed account on these rings and on $p$-adic Hodge theory 
we refer for example to \cites{MR1922833,fontaine_ouyang}. 
\begin{enumerate}
\item \label{Bdr_1}
\Bdr\ is a complete discrete valuation field, its ring of integer is denoted by $\Bdr^+$ and the 
 residual field identifies to  the $p$-adic completion $\Cp$ of $\Qpbar$. 
\item \label{Bdr_2}
\Bdr\ is filtered by the (fractional) powers of the maximal ideal of $\Bdr^+$, and we denote this filtration by 
$(\Fil[i]\Bdr)_{i\in\ZZ}$. The map $\vdR$ defined in Convention \ref{sb:Conv_Fil} on $\Bdr$  is its discrete valuation.

\item \label{Bdr_3}
$\Bdr$ is endowed with an action of $G_{\Qp}$.
We have canonical inclusions $\Qpbar\subset \Bdr$ and $\widehat{\QQ}_p\nr \subset \Bdr$ compatible with the action of $G_{\Qp}$.

\item \label{Bcris_1}
$\Bcris \subset \Bdr$ is a sub-$\widehat{\QQ}_p\nr$-algebra stable under the $G_{\Qp}$-action.

\item \label{Bcris_2}
We have an endomorphism $$\Fro[\mathrm{cris}]\colon \Bcris\rightarrow \Bcris,$$ semilinear with respect to the absolute Frobenius $\Fro$ of $\widehat{\QQ}_p\nr$. We call it the Frobenius of $\Bcris$ and denote it simply by $\Fro$.

\item \label{Bcris_3}
There exists an element $t\in \Bcris$, such that $$\Fro(t)=pt\quad \text{ and }\quad t\in\Fil[1]\Bdr\,
 \text{ (actually }\Fil[1]\Bdr= t\Bdr^{+}\text{).}$$ 
For any $g$ in $G_{\Qp}$, we have $$g(t)=\chi(g)t,$$ where $\chi\colon G_{\Qp}\rightarrow \Zp^{\times}$ denotes the cyclotomic character.
We set $\Qp(1)\coloneqq\Qp\cdot t\subset\Bcris$. %
\end{enumerate}

 Let $L\subset \Qpbar$ be a  finite extension over $\Qp$. 
Keep the notation of Convention \ref{sb:Conv_p-adic_fields}. 

\begin{enumerate}[resume]
\item \label{BcrisE1}
We will denote  by $\BcrisL$ the smallest subring of \Bdr\ containing \Bcris\ and $L$. 
We recall that the natural map 
$$\Bcris\otimes_{\La}L\xrightarrow{\,\simeq\,} \BcrisL\subset \Bdr$$
is an isomorphism. 
\item \label{BcrisE3}
The Frobenius of $\Bcris$ does not extend to $\BcrisL$ in general, nevertheless its power  $\Fro[\mathrm{cris}]^{f_L}$ is 
$\La$-linear, hence we can extend it to \BcrisL\ by 
$ \Fro[\mathrm{cris}]^{f_L} \otimes \id[L]$. We will denote it still by $\Fro[\mathrm{cris}]^{f_L}$  or $\Fro^{f_L}$. Again, we may drop the subscript $L$ in the notation if there is no risk of confusion.

\end{enumerate}


\section{Main statements} \label{sect:MS}

In this section we state our main result  (Theorem \ref{mainthm}) and  then give some geometric consequences. To put the result into perspective, we formulate a conjecture (Conjecture \ref{mainconj}) and discuss the role of the different hypothesis (Remark \ref{rem hypo}).

\begin{conj}\label{mainconj}
Let $K$ be a $p$-adic field, $\mathcal{O}_K$  its ring of integers and $k$  its residue field.
Let $(M,q)$ be a motive in mixed characteristic \[M \in \CHM(\mathcal{O}_K)_\QQ\] 
together with a map
$q\colon \Sym^2 M  \rightarrow \one.$
 Write $\cdot_{/_{k}}$ for the restriction functor  to the category $\CHM(k)_\QQ$ and $R_B$ for the Betti realization induced by a fixed embedding   $\sigma : K \hookrightarrow \CC$.
 
Define   two $\QQ$-quadratic  spaces $(V_B,q_B)$ and $(V_Z,q_Z)$    as follows
\[(V_B,q_B) \coloneqq R_B(M, q ),  \hspace{0.5cm}
V_Z \coloneqq \Hom{\NUM(k)_{\QQ}}{\one}{ M_{/_{k}}},\] 
\[\text{ and} \hspace{0.5cm} q_Z\colon V_Z \rightarrow \End{}{\one}=\QQ, \hspace{0.5cm} \text{defined by} \hspace{0.5cm} z\mapsto q_{/_{k}} \circ \Sym^2(z). \]
	

We conjecture that if   $q_B$   is a polarization of Hodge structures then  $q_Z$  is positive definite.
\end{conj}

\begin{theorem}\label{mainthm}
Let $M,q,q_B$ and $q_Z$ be as in Conjecture $\ref{mainconj}$.

Let $F$ be a number field which acts on $M$ and which is endowed with  a non trivial involution $*$.
Assume the following:
\begin{enumerate}
\item \label{mainthm:1} The equality  $ \dim_{\QQ} V_Z = \dim M$ holds.

\item \label{mainthm:2} The equality  $ [F:\QQ] = \dim M$ holds.

\item \label{mainthm:3} For every $\cf$ in $F$ the adjoint of $\cf$ with respect to $q$ is $\cf^*$.

\item \label{mainthm:4} The pairing $q_B$ on $V_B$ is a polarization of Hodge structures.

\end{enumerate}

Then the signature $(s_+,s_-)$ of the quadratic form $q_Z$ satisfies $4 \vert s_-$.

\end{theorem}
\noindent The proof of Theorem \ref{mainthm} is given is Section \ref{sect:proof_MT}.

\begin{remark} \label{rem hypo}
Hypothesis (4) is  crucial and absolutely necessary. Indeed, in order to conclude a positivity statement one has to assume some positivity property.
The action of $F$ makes the hypothesis (1)  essentially automatic. Indeed if $ \dim V_Z \neq \dim M$ then   $V_Z=0$ and the statement is trivial. See  \cite[proof of Proposition 6.8]{MR4199442}
 for details.

The statement should be true without the hypothesis (2) and (3), see Conjecture \ref{mainconj}.
On the other hand, the standard conjecture of Hodge type can be reduced to the case of varieties over finite fields  \cite[Proposition 3.16]{MR4199442} and for such a variety  the Tate conjecture predicts that its motive should be a motive of abelian type \cite{MR1265538}. In particular the hypothesis of Theorem \ref{mainthm} should not be restrictive in the study of  the standard conjecture of Hodge type.

Finally,  homological and numerical equivalence should always coincide hence one would like to replace $V_Z$ with the space of cycles modulo homological equivalence. The following proposition is the crucial reason for which one needs to work with numerical equivalence.
\end{remark}

\begin{lemma}\label{supersing}
Keep notation from Theorem $\ref{mainthm}$. Then the hypothesis $(1)$ is equivalent to the existence of an isomorphism 
$M_{/_{k}}\cong \one^{\oplus \dim M}$ of homological motives.
In particular $V_B$ and all classical realizations of $M$ have weight zero and if $E$ is the field of coefficients of a given realization $R$, then
\[  R(M_{/_{k}})= V_Z \otimes_{\QQ} E,  \hspace{0.5cm}  \textrm{and} \hspace{0.5cm} R(q_{/_{k}})= q_Z \otimes_{\QQ} E. \]
\end{lemma}
\begin{dem}
As numerical motives form a semisimple category ({\itshape cf.} \cite{MR1150598}), the hypothesis $(1)$ is equivalent to the existence of an isomorphism 
$M_{/_{k}}\cong \one^{\oplus \dim M}$ of numerical motives. Let us fix such an isomorphism $f: M_{/_{k}}\rightarrow \one^{\oplus \dim M}$   and let $g$ be its inverse. Consider a map $\tilde{f}: M_{/_{k}}\rightarrow \one^{\oplus \dim M}$ at the level of homological motives whose reduction modulo numerical equivalence is $f$ and similarly consider $\tilde{g}$ a lifting of $g$.

The ring of endomorphisms of the unit object $\one$ is $\QQ$, both under numerical or homological equivalence. In particular, passing from homological to numerical equivalence in such an endomorphisms ring does not kill any nonzero map. 
Hence, the composition $\tilde{f} \circ \tilde{g}$  is the identity, as so is $f \circ g$.
This implies that the map  $\tilde{g} \circ \tilde{f}$ is a projector inducing a decomposition of the form
$M_{/_{k}} \iso \one^{\oplus \dim M} \oplus N$ for some homological motive $N$.
For dimensional reasons, the realization of $N$ is zero, and so is $N$ by the very definition of homological motives.

 Conversely, the existence of an homological isomorphism   
$M_{/_{k}}\cong \one^{\oplus \dim M}$  clearly implies hypothesis $(1)$.    
%
\end{dem}
Let us now conclude the section with some applications of Theorem \ref{mainthm}.

\begin{cor}\label{cor:AV}
Let $A$ be an abelian variety over a finite field, $\widetilde{A}$ be a CM-lifting, $\widetilde{L}$ be a hyperplane section of $\widetilde{A}$ and  
  $L$ be the restriction of  $\widetilde{L}$ to $A$. For  a positive integer $n\leq \dim(A)/2$,
 let $(s_+,s_-)$ be the  signature of the quadratic form $q_{A,L,n}$ from Definition $\ref{def intro}$, then  $s_-$ is divisible by $4$.
\end{cor}
\begin{dem}
We argue as in  \cite[\S 8]{MR4199442}
 and use the complex multiplication to decompose the motive of $A$ in an orthogonal sum of smaller motives. Among these factors, the ones we need to study are the so called exotic. They all fit in the hypothesis of Theorem \ref{mainthm}.
\end{dem}

\begin{remark}
	A priori Corollary \ref{cor:AV} does not apply to a hyperplane sections $L$ which does not lift to characteristic zero. Nevertheless it is possible  
	sometimes to extend it to all hyperplane sections, as for example  in  \cite[Proposition 3.15]{MR4199442} and \cite[Lemma 2.2]{Koshi_HSC_prime}. 
\end{remark}

\begin{cor}
Let $X=S_1\times   \ldots \times S_m$ be the product of  K3 surfaces over  a finite field, $L_i$ be a hyperplane section induced, as in Corollary \ref{cor:AV}, by a  CM lifting of $S_i$, 
  and let $L=\oplus_i p^*_i L_i$ be the induced hyperplane section on $X$. Let $n \leq \dim(X)/2 =m$ be any integer and   
  $(s_+,s_-)$ be the  signature of the quadratic form $q_{X,L,n}$ from Definition $\ref{def intro}$, then  $s_-$ is divisible by $4$.
\end{cor}
\begin{dem}
The argument follows the same lines as in  \cite[\S 8]{MR4199442}, namely we use the complex multiplication to decompose the motive into summands to which we can apply Theorem \ref{mainthm}.

First, each  $S_i$ has a CM lifting $\tilde{S}_i$ by \cite{MR4241794}.  By \cite{Kahn}, the motive $\mathfrak{h}(\tilde{S}_i)$ of $\tilde{S}_i$ admits a decomposition
\[\mathfrak{h}(\tilde{S}_i)=  \one \oplus  \one(-2) \oplus  \one(-1)^{\oplus \rho_i} \oplus  \mathfrak{h}^{t}(\tilde{S}_i) \]
where $\rho_i$ is the Picard number of $\tilde{S}_i$ and $\mathfrak{h}^{t}(\tilde{S}_i)$ is the motive whose realization is the transcendental part  $H^{2,t}(\tilde{S}_i)$ of the Hodge structure $H^2(\tilde{S}_i)$.

Now by \cite[Corollary 1.3]{MR4015230}, there is a CM field $L_i$ acting on $\mathfrak{h}^{t}(\tilde{S}_i)$ with the property that the $\mathbb{Q}$-dimensions of $L_i$ and $H^{2,t}(\tilde{S}_i)$ are the same. Hence, after extending the scalars to $\overline{\mathbb{Q}}$,  the motive $\mathfrak{h}^{t}(\tilde{S}_i)$ decomposes into a sum of motives of rank one, see   \cite[Proposition 6.6]{MR4199442}.
By the Künneth formula, such a decomposition in motives of rank one holds true for the whole motive of $X$, see  \cite[Proposition 6.7(1)]{MR4199442}. 
By taking the orbits under the  Galois action on the coefficients, we can  
descend such a decomposition into a decomposition with $\QQ$-coefficients, see  \cite[Proposition 6.7(2)]{MR4199442}. Notice that the factors are not anymore  of rank one in general but 
the  decomposition is orthogonal,  hence it is enough to work with a single factor $M$. 

 Assume that $M_{/_{k}}$ contains an algebraic class which is nonzero modulo numerical equivalence (otherwise the factor has no interest for the quadratic form $q_{X,L,n}$). Then $M$ satisfies the hypothesis \eqref{mainthm:1} of Theorem \ref{mainthm}, see \cite[Proposition 6.8]{MR4199442}. 
  By construction the hypotheses \eqref{mainthm:2} and \eqref{mainthm:3} are verified for the motives $\mathfrak{h}^{t}(\tilde{S}_i)$ and  one can check that 
they  are stable under the tensor operations above. 
Finally the hypothesis \eqref{mainthm:4} comes from the Hodge--Riemann bilinear relations. We can then apply  Theorem \ref{mainthm} and deduce the desired conclusion for the factor $M$. 
\end{dem}
\begin{remark}
There are probably other varieties to which Theorem \ref{mainthm} can be applied. As already mentioned in Remark \ref{rem hypo}, all motives over a finite field should be of abelian type, hence they should have a CM lifting \cite{MR1265538}. Some other examples for which this is known are cubic Fermat hypersurfaces \cite{MR552586}. In order to apply Theorem \ref{mainthm} to them one has to check that the decomposition induced by the CM action is orthogonal with respect to the quadratic form $q_{X,L,n}$. Similarly, one might try to study Kummer varieties or the Hilbert scheme of points on a K3.
\end{remark}

\section{From archimedian to $p$-adic} \label{sect:Arch2padic}

We recall here some classical facts on quadratic forms. They will allow us to reduce Theorem \ref{mainthm} to a $p$-adic statement. 
In what follows $\QQ_\nu$ denotes the completion of $\QQ$ at the place $\nu$. Recall that at every place one defines  $\varepsilon_\nu(q) \in \{\pm 1\},$  the Hilbert symbol (or Hasse symbol) of a non-degenerate $\QQ$-quadratic form $q$ at $\nu$, {\itshape cf.} \cite[Ch.~IV, \S2.1 and \S2.4]{MR0498338}.

\begin{prop}\label{remarksignature} \cite[Ch.~IV, \S2.4]{MR0498338}.
Let $q$ be a non-degenerate $\QQ$-quadratic form and let $(s_+, s_-)$ be its signature. Then the discriminant is positive if and only if $2 \vert s_-.$ In this case,  $4 \vert s_-$ if and only if    $\varepsilon_\RR(q)=+1$.
\end{prop}

\begin{prop}\label{localinvariant} {\cite[Ch.~IV, \S 2.3, Theorem 7]{MR0498338}}. 
Two non-degenerate 
 $\QQ_p$-quadratic forms $q_1$ and $q_2$ of same rank are 
isomorphic 
if and only if the discriminants of $q_1$ and $q_2$ coincide  and 
$\varepsilon_p(q_1)=\varepsilon_p(q_2)$.
\end{prop}

\begin{theorem}\label{productformula} {\cite[Ch.~IV, \S 3.1]{MR0498338}}. 
Let $q$ be non-degenerate $\QQ$-quadratic form. Then for all but a finite number of places $\nu$ the equality $\varepsilon_\nu(q)=+1$ holds. Moreover, the following product formula running on all places holds
\[\prod_\nu  \varepsilon_\nu(q)= +1.\]
\end{theorem}

\begin{cor}\label{signaturetop}
Let $q_1$ and $q_2$  be two non-degenerate $\QQ$-quadratic forms and let $p$ be a prime number. Suppose that, for all primes $\ell$ different from $p$, we have
\[q_1 \otimes \QQ_\ell \cong q_2 \otimes \QQ_\ell.\]
Then the two quadratic forms $q_1$ and $q_2$ have the same discriminant. Moreover, if this discriminant is positive, then $4\vert s_-(q_1)$ if and only if the equality
\[\varepsilon_p(q_1)=\varepsilon_\RR(q_2)\varepsilon_p(q_2).\]
holds.
\end{cor}
\begin{proof}
The $\ell$-adic hypothesis implies in particular that the discriminants of $q_1$ and $q_2$ coincide in $ \QQ_\ell^{\times}/(\QQ_\ell^{\times})^2$ 
 for all   $\ell\neq p$. This implies that they coincide in $ \QQ^{\times}/(\QQ^{\times})^2$
by \cite[Theorem 3 in 5.2]{MR1070716}.

 Suppose from now on that the discriminant is positive. By Proposition \ref{remarksignature}, $4\vert s_-(q_1)$   if and only $\varepsilon_\RR(q_1)=+1.$

\smallskip

Now, Theorem \ref{productformula} implies that
$\prod_\nu  \varepsilon_\nu(q_1)=1=\prod_\nu  \varepsilon_\nu(q_2).$
Combining this with the $\ell$-adic isomorphisms we deduce
\[\varepsilon_\RR(q_1)\varepsilon_p(q_1)=\varepsilon_\RR(q_2)\varepsilon_p(q_2).\] This means that  $\varepsilon_\RR(q_1)=+1$ if and only if $\varepsilon_p(q_1)=\varepsilon_\RR(q_2)\varepsilon_p(q_2)$. \end{proof}

\begin{prop}\label{padicreduction}
Let us keep notation from Theorem $\ref{mainthm}$. 
Let $p$ be the characteristic of $k$ and $h_i$ be the dimension of the $(i,-i)$-part of the Hodge structure $V_B$. 
Define the positive integer \[s_M \coloneqq \sum_{i \geq 1,\,\text{odd}}  h_i.\]
Then the quadratic forms $q_B$ and $q_Z$ have the same discriminant, which is positive. Moreover, Theorem  $\ref{mainthm}$ holds true if and only if 
the equality
\[\varepsilon_p(q_Z)/\varepsilon_p(q_B)=(-1)^{s_M}.\]
holds true.
\end{prop}
\begin{proof}
We want to apply Corollary \ref{signaturetop} to $q_1=q_Z$ and $q_2=q_B$.

First of all, by  Lemma \ref{supersing}, we have that $q_Z \otimes \QQ_\ell = R_\ell(q_{/_{k}})$. By Artin comparison theorem, we have $q_B \otimes \QQ_\ell = R_\ell(q_{/_{\CC}})$. Combining these equalities with smooth proper base change in $\ell$-adic cohomology   we deduce that \[q_B \otimes \QQ_\ell = q_Z \otimes \QQ_\ell.\]
Now, by hypothesis $q_B$ is a polarization for the Hodge structure $V_B$ and recall that $V_B$ has weight zero (Lemma \ref{supersing}). In particular, the Hodge--Riemann relations compute the signature $(s_{B,+},s_{B,-})$ of $q_B$ as
  \[s_{B,+}=\sum_{i\, \text{even}} h_i \hspace{0.5cm}\textrm{and} \hspace{0.5cm} s_{B,-}=\sum_{i\, \text{odd}} h_i.\]
Note that $s_{B,-}$ is even because of the Hodge symmetry.
This implies through Proposition \ref{remarksignature} that  the discriminant is positive and   that the Hilbert symbol of  $q_B\otimes \RR$ is $(-1)^{\sum_{i\geq 1, \text{odd}}h_i}$.  We can now conclude by applying Corollary \ref{signaturetop} to $q_1=q_Z$ and $q_2=q_B$.
\end{proof}

\begin{remark}\label{qualcosa}
We have ${\sum_{i\geq 1, \text{odd}}h_i}\equiv {\sum_{i\geq 0} ih_i}$ (mod $2$), in particular 
the integer $s_M$ has the same parity as  the minimum of the Hodge polygon. 
\end{remark}
The role of the above proposition is to translate a positivity problem (Theorem \ref{mainthm}) into a $p$-adic problem. The advantage is that the latter has a cohomological interpretation, as the following proposition shows.
\begin{prop}\label{motive admissibility}
Let us keep notation from Theorem $\ref{mainthm}$ and Convention $\ref{sb:Conv_Realis}$.
Then the  following holds.
\begin{enumerate}
\item \label{ma1}
The quadratic space $(V_B,q_B)\otimes_\QQ \Qp$ is isomorphic to $(R_p(M),R_p(q))$.
\item \label{ma2}
The Galois representation $R_p(M)$ is crystalline and the admissible filtered-$\varphi$-module which corresponds to it is 
$\Rcris(M)$.
\item \label{ma3}
The $\Qp$-subspace of Frobenius invariant vectors in $\Rcris(M)$ identifies with $V_Z  \otimes_\QQ \Qp$ and generates $\Rcris(M)$. 
\item \label{ma4}
The quadratic form $q_Z  \otimes_\QQ \Qp$ identifies with the restriction of $\Rcris(q)$ to $V_Z  \otimes_\QQ \Qp$.
\end{enumerate}
\end{prop}
\begin{proof}
Point \eqref{ma1} comes from Artin's comparison theorem. Point \eqref{ma2} comes from Falting's comparison theorem in $p$-adic Hodge theory \cite{MR1463696}, see also \cite{MR4199442}*{Theorem 10.2}. For part \eqref{ma3}, notice that in general algebraic classes are included in the Frobenius invariant ones, but for dimensional reasons (hypothesis \eqref{mainthm:1} in Theorem  \ref{mainthm}) this inclusion is an equality. Point \eqref{ma4} follows from
 Lemma~\ref{supersing}.
\end{proof}

\section{Reduction to CM-simple objects} \label{sect:SObj}

In this section we study quadratic forms in the context of $p$-adic Hodge theory.
The setting we choose is inspired by the previous section and in particular by Propositions \ref{padicreduction}  and \ref{motive admissibility}.

\begin{definition}\label{def supersingular}
A supersingular pair   $( V_{B,p}, V_{Z,p})$ is the collection of the following objects.
\begin{enumerate}
\item A  $\Qp$-vector space  $V_{B,p}$ endowed with the action of the absolute Galois group of a $p$-adic field $K$ which makes the representation crystalline.
\item A $\Qp$-vector space $V_{Z,p}$ which generates the (admissible) filtered $\varphi$-module corresponding to $V_{B,p}$ and on which the Frobenius acts trivially.
\end{enumerate}

An orthogonal supersingular datum   $( V_{B,p}, V_{Z,p},q_{B,p}, q_{Z,p} )$, or simply an orthogonal supersingular   pair $( q_{B,p}, q_{Z,p} )$ is the collection of the following objects.
\begin{enumerate}

\item A supersingular pair   $( V_{B,p}, V_{Z,p})$.
\item A non-degenerate $\Qp$-quadratic form
$q_{B,p}$ on $V_{B,p}$ for which the Galois action is isometric. 
\item The non-degenerate $\Qp$-quadratic form
$q_{Z,p}$ on $V_{Z,p}$ corresponding to $q_{B,p}$.
\end{enumerate}
\end{definition}
\begin{lemma}\label{lemma discriminant}
Let $( q_{B,p}, q_{Z,p} )$ be an orthogonal supersingular pair. Then the discriminants of $q_{B,p}$  and $q_{Z,p} $ coincide.
\end{lemma}
 
\begin{proof}
The quadratic form on $V_{Z,p}$ induces a quadratic form on the one dimensional vector space $\det V_{Z,p}$. This form is characterized by an element of $\QQ_p^{\times}/(\QQ_p^{\times})^2$, which is the discriminant of $q_Z$. (The same holds for  $V_{B,p}$.) As the Frobenius acts trivially on  $\det V_{Z,p}$, admissibility implies that the filtration is trivial as well. This means that the comparison isomorphism sends the $\QQ_p$-structure  $\det V_{Z,p}$ on the $\QQ_p$-structure  $\det V_{B,p}$. As the comparison isomorphism respects the underlined quadratic forms the discriminants coincide.  \end{proof}

\begin{lemma}\label{key lemma}
Let   $( V_{B,p}, V_{Z,p},q_{B,p}, q_{Z,p} )$ be an orthogonal supersingular datum, let   $F$ be a number field and let $*$ be a non trivial field involution on $F$. Then the following holds true.
\begin{enumerate} 
\item \label{key lemma_1}
The ring $F \otimes_{\QQ} \QQ_p$ is a product of $p$-adic fields
\[ F \otimes_{\QQ} \QQ_p= \prod_{i=1}^n F_i \] 
\item \label{key lemma_2}
The involution $*$ acts on this product as a composition of disjoint transpositions, i.e. there is an integer $s$ such that, after changing the numbering, one has
 \[*(F_i)=F_{i+s}\hspace{0.2cm} \textrm{for} \hspace{0.2cm}  i\leq s, \hspace{0.2cm} *(F_i)=F_{i-s}  \hspace{0.2cm} \textrm{for} \hspace{0.2cm} s < i\leq 2s\hspace{0.2cm}  \textrm{and} \hspace{0.2cm} *\!(F_i)=F_{i}  \hspace{0.2cm} \textrm{for}\hspace{0.2cm}  i > 2s.\]
 Moreover the induced involution on $F_{i}$ for $ i > 2s$ is non trivial.
 \item \label{key lemma_3}
Suppose that $F$ acts on $V_{B,p}$ and that this action commutes 
   with the Galois action. Then $F$ acts on $V_{Z,p}$ as well and the two actions are compatible with respect to the comparison theorem. In particular the decomposition of point (1) induces two decompositions \[V_{B,p}= \bigoplus_{i=1}^n V^i_{B,p}\ \hspace{0.5cm} \textrm{and} \hspace{0.5cm}V_{Z,p}= \bigoplus_{i=1}^n V^i_{Z,p}\] 
   and for all $i$ the pair $( V^i_{B,p}, V^i_{Z,p})$  is a  supersingular pair.
 \item \label{key lemma_4}
Suppose moreover that for all $\cf\in F$ the adjoint of $\cf$ with respect to $q_{B,p}$ is $\cf^*$. Then the same holds true for $q_{Z,p}$. Moreover the decompositions in $n-s$ factors
  \[V_{B,p}= \bigoplus^\perp_{s < i\leq 2s} (V^i_{B,p}\oplus V^{i-s}_{B,p}) \oplus^\perp \bigoplus^\perp_{i>2s} (V^i_{B,p}) \hspace{0.5cm} \textrm{and} \hspace{0.5cm} V_{Z,p}= \bigoplus^\perp_{s < i\leq 2s} (V^i_{Z,p}\oplus V^{i-s}_{Z,p}) \oplus^\perp \bigoplus^\perp_{i>2s} (V^i_{Z,p})\] 
 are orthogonal. 
 \item \label{key lemma_5}
Keep notation and hypothesis from point (4).  For all $i>s$ denote by $q^i_{B,p}$ the restriction of $q_{B,p}$ to each orthogonal factor of the above decomposition and similarly for  $q^i_{Z,p}$. Then the pair   $(q^i_{B,p},  q^i_{Z,p})$ is an orthogonal supersingular pair.
 \item \label{key lemma_6}
For  $s < i\leq 2s$ consider   the orthogonal supersingular pair $(q^i_{B,p},  q^i_{Z,p})$ defined in point (5). Then  the two $\Qp$-quadratic forms $q^i_{B,p}$ and $  q^i_{Z,p}$ are isomorphic. More  precisely, the two subspaces $V^i_{B,p}$ and $V^{i-s}_{B,p}$ of  $V^i_{B,p}\oplus V^{i-s}_{B,p}$  are maximal isotropic  for the quadratic form $q^i_{B,p}$ (and similarly for $q^i_{Z,p}$). Moreover $q^i_{B,p}$ realizes  the Galois representations $V^i_{B,p}$ and $V^{i-s}_{B,p}$ as one the dual of the other.
  \end{enumerate}
\end{lemma}

\begin{proof}
Point \eqref{key lemma_1} comes from the fact that $F/\QQ$ is a separable extension. 

For point \eqref{key lemma_2} consider $F \otimes_{\QQ} \QQ_p= F_1\times \ldots \times F_n$. Its spectrum is a disjoint union of $n$ closed points and $*$ acts on it as a permutation of order two. Consider now $ i > 2s$ as in the statement and the   inclusion $F \hookrightarrow F_i$ induced from point \eqref{key lemma_1}. By construction the involution on $F_i$ extends the one on $F$, hence it is non trivial.

For point \eqref{key lemma_3}, by the comparison theorem, as $F$ acts on $V_{B,p}$ seen as crystalline Galois representation of $G_K$, then  $F$ must act also on its corresponding filtered $\varphi$-module. This means that the action of $F$ must commute with the Frobenius hence $F$ stabilizes the Frobenius invariant part, i.e. $V_{Z,p}$.

Let us now study point \eqref{key lemma_4}, let $q$ be one of the quadratic forms we want to study and $b$ be its corresponding bilinear form. Let $p_i$ be the $i$-th projector in the decomposition $F \otimes_{\QQ} \QQ_p= F_1\times \ldots \times F_n$. By hypothesis we have $b(p_i(-), p_j (-))= b((-), p_i^* p_j (-))$. But $p_i^* p_j = 0$ except if $p_j=p_i^*$.

Point \eqref{key lemma_5} is automatic from the construction. For point \eqref{key lemma_6} we have to study the quadratic forms on the space $V^i_{B,p}\oplus V^{i-s}_{B,p}$ for $s < i\leq 2s$ (and similarly for $V^i_{Z,p} \oplus V^{i-s}_{Z,p}$). Notice that the proof of point \eqref{key lemma_4} we gave, shows in particular that the two subspaces $V^i_{B,p}$ and $V^{i-s}_{B,p}$ are maximal isotropic. 
Thus the quadratic forms $q_{B,p}$ and $q_{Z,p}$ are both isomorphic to a sum of hyperbolic planes. 
The duality statement comes from the fact that $q_{B,p}$ is non-degenerate by hypothesis.
\end{proof}

\begin{definition}\label{definition good}
Let $M=( q_{B,p}, q_{Z,p} )$
be an orthogonal supersingular pair, let $\varepsilon_p(q_{B,p})$ and  $\varepsilon_p(q_{Z,p})$  be the Hilbert symbols of the quadratic forms and define  
  $s_M$ to be the minimum of the Hodge polygon of the filtered $\varphi$-module underlying $M$.  We say that $M$ is good if the following equality holds
\begin{equation}\label{hilbert formula}\varepsilon_p(q_{Z,p})/\varepsilon_p(q_{B,p})=(-1)^{s_M}.\end{equation}
\end{definition}
\begin{lemma}\label{buona notte}
A supersingular pair $M=( q_{B,p}, q_{Z,p} )$ is good if and only if one of the following situation happens:
\begin{enumerate}
\item The quadratic forms $q_{B,p}$ and $q_{Z,p}$ are isomorphic and $s_M$ is even.
\item The quadratic forms $q_{B,p}$ and $q_{Z,p}$ are not isomorphic and $s_M$ is odd.
\end{enumerate}
\end{lemma}
\begin{proof}
 The two quadratic spaces have the same rank by definition of supersingular pair 
and the same discriminant  by Lemma \ref{lemma discriminant}.  Hence
 they are isomorphic if and only if they have the same Hilbert symbol (Proposition \ref{localinvariant}). This translates the definition of good to the conditions (1) or (2) of the statement.
\end{proof}
\begin{prop}\label{prop simple object}
Keep notation as in Lemma $\ref{key lemma}$.
If for each  $i>2s$   the pair   $(q^i_{B,p},  q^i_{Z,p})$ is good then the pair   $(q_{B,p},  q_{Z,p})$ is good as well.
 \end{prop}
\begin{proof}
First of all, let us show that the pair $(q^i_{B,p},  q^i_{Z,p})$ is good also for $s < i\leq 2s$. In this case the two quadratic forms are isomorphic by Lemma \ref{key lemma}\eqref{key lemma_6}, hence $\varepsilon_p(q^i_{Z,p})/\varepsilon_p(q^i_{B,p})=1.$ We need to show that the height of the corresponding Hodge polygon is even.

Let $  D_{Z,p}^i$ 
 be the admissible filtered $\Fro$-module associated with the factor $V^i_{Z,p}$.  The height of the Newton polygon of each one of those modules is zero 
by construction as they have a Frobenius invariant basis. In particular, since they are admissible, the height of the Hodge polygon is also zero. Moreover, by Lemma \ref{key lemma}\eqref{key lemma_6}, $D_{Z,p}^i $ and  $D_{Z,p}^{i-s}$ are  dual to each other. This duality together with   height zero implies that the two Hodge polygons are symmetric to each other.
In particular, they have the same minimum, hence the minimum of the Hodge polygon of   $D_{Z,p}^i \oplus D_{Z,p}^{i-s}$ is even.

 Now, as the right hand side of the equation \eqref{hilbert formula}  is additive on direct sums, let us show that the left hand side is additive as well, this will imply the statement. Consider two orthogonal supersingular pairs $M_1$ and $M_2$ which are both partial sums of the factors in the decomposition of Lemma \ref{key lemma}\eqref{key lemma_4}.
Let $\delta_{B,1}, \eps_{B,1},\delta_{Z,1}, \eps_{Z,1}$ be the discriminants and the Hilbert symbols of the two quadratic forms associated with  $M_1$ and similarly let $\delta_{B,2}, \eps_{B,2},\delta_{Z,2}, \eps_{Z,2}$ be the invariants for  $M_2$.
By the definition of the Hilbert symbol of a quadratic form \cite[Ch.~IV, \S2.1]{MR0498338}, it follows that the Hilbert symbols of the sum $M_1\oplus M_2$ are respectively
\[ \eps_B(M_1\oplus M_2) = \eps_{B,1}  \cdot   \eps_{B,2}\cdot (\delta_{B,1}, \delta_{B,2})\]
and 
\[ \eps_Z(M_1\oplus M_2) = \eps_{Z,1}  \cdot   \eps_{Z,2}\cdot (\delta_{Z,1}, \delta_{Z,2}),\]
where $(\delta_{1}, \delta_{2})$   denotes the Hilbert symbol of a couple of elements 
 $\delta_{1}$, $\delta_{2} \in\Qp^{\times}/(\Qp^{\times})^2$, {\itshape cf. }\cite[Ch.~III, \S1.1, Th. 2]{MR0498338}.
%

If we write the quotient of the two equalities above the contribution of the discriminants simplifies as $\delta_{B,1}= \delta_{Z,1}$ and  $\delta_{B,2}= \delta_{Z,2}$ by Lemma \ref{lemma discriminant}.
\end{proof}

\section{CM-quadratic forms over $p$-adic fields}\label{sect:CQ-qf} 

From now on (except for Section \ref{sect:proof_MT}) let  $F$   be a finite extension of $\Qp$.
We assume that $F$ is endowed with a non trivial involution $*\colon F\rightarrow F$, $\alpha\mapsto \alpha^*$. We collect here some basic results on some quadratic forms endowed with an action of $F$ which we call CM 
(Definition \ref{def CM quad form}). 

Throughout the section we fix a Galois closure $\tilde{F}\supset F$   of $F$, and denote by $\Gamma$ the set of $\Qp$-embeddings of $F$ in $\tilde{F}$. 
We recall that $F_0$ denote the subfield of $F$ where $*$ acts as the identity, {\itshape cf.} Convention \ref{choucroute}.
We denote still by $*$ the action of $*$ on $\Gamma$ by pre-composition.

Denote respectively by $\tr_{F_0/\Qp}\colon F_0\rightarrow \Qp$ and $N_{F/F_0}\colon F\rightarrow F_0$ the trace of the extension $F_0/\Qp$ and the norm of $F/F_0$; the extension $F/F_0$ has degree two and we have $N_{F/F_0}(x)=x^*x$ for every $x\in F$. 

\begin{definition}\label{def CM quad form}
 
 \begin{enumerate}
 \item A CM-space (with respect to $F$) is the data of a $\Qp$-vector space $V$ and an action 
 of $F$ on $V$ such that $V$ as an $F$-vector space   has dimension one.

\item
A CM-quadratic space (with respect to $F$) is the data of a CM-space $V$ and a non-degenerate
quadratic form $q\colon V\rightarrow \Qp$ on the $\Qp$-vector space $V$, such that
for all $\alpha$ in $F$, the adjoint with respect to $q$ of the multiplication by $\alpha$ is the multiplication by $ \alpha^*$.   
\end{enumerate}
For simplicity we may say that $q$ (or $V$) is a $CM$-quadratic form without mentioning explicitly the other structures.
\end{definition}

\begin{prop} Let $V$ be a CM-quadratic space.
Fix a nonzero vector $v$ in $V$ and consider the induced identification $V \cong F$ as $F$-vector spaces. Under this identification,
there exists a unique $a\in F_0^{\times}$ such that for every $x\in F$, we have 
\[q(x)=\tr_{F_0/\Qp}(a N_{F/F_0}(x)).\]

If one changes the choice of the vector $v$, the element $a\in F_0^{\times}$ is multiplied by some norm in $N_{F/F_0}(F^{\times})$. In particular the class of $a$ in $F^{\times}_0/N_{F/F_0}(F^{\times})$ is well-defined and depends only on $V$. 
\end{prop} 
\begin{dem}
Let $b\colon F\times F\rightarrow \Qp$ the symmetric form associated with $q$.
Since $q$ is a CM-quadratic form, we have, for every $x\in F$, $q(x)=b(x,x)=b(1,x^*x)=b(1,N_{F/F_0}(x))$.
The restriction of $b(1,-)$ to $F_0$ gives a $\Qp$-linear form on $F_0$. Since $(u,v)\mapsto \tr_{F_0/\Qp}(uv)$ is a non-degenerate bilinear form on $F_0$ we can conclude that there exists such an element $a\in F_0$. It must be nonzero otherwise $q$ would be zero.

Let $w$ be a new nonzero vector in $V$ and $q'(x)$ be the induced quadratic form on $F$ under this new identification. There is a unique nonzero element $\cf\in F$ such that $\cf v=w$. This gives the relation $ q'(x)=q(\cf x)$. Hence one deduces 
\[q'(x)=\tr_{F_0/\Qp}(a N_{F/F_0}(\cf x)) = 
\tr_{F_0/\Qp}((a\cf^*\cf ) \cdot N_{F/F_0}(x)).\qedhere\]
\end{dem}
\begin{definition}
Let $q$ be a CM-quadratic form. We define the gauge of $q$ as the element of  $  F^{\times}/N_{F/F_0}(F^{\times})$   associated with $q$ through the above proposition. 
\end{definition}

\begin{theorem}[Milnor]\label{thm milnor}~
\begin{enumerate}
	\item\label{thm milnor:1} The group $  F^{\times}/N_{F/F_0}(F^{\times})$  has cardinality two.
		\item\label{thm milnor:x} All CM-quadratic forms (with respect to the same $F$) have the same discriminant.
	\item\label{thm milnor:2} Two CM-quadratic forms are isomorphic if and only if they have the same gauge in $  F^{\times}/N_{F/F_0}(F^{\times})$.
	\item\label{thm milnor:3} Up to isomorphism there are exactly two CM-quadratic spaces. 
\end{enumerate}
\end{theorem}
\begin{dem} 
The extension $F/F_0$ has degree $2$. The reciprocity isomorphism $$F^{\times}/N_{F/F_0}(F^{\times})\iso \Gal{F}{F_0}$$
 of local class field theory  proves \eqref{thm milnor:1}. 
 
Consider  $F$ as an $F_0$-vector space of dimension 2 and consider on it the non-degenerate $F_0$-quadratic forms  $\tilde{q}_a(x)=a N_{F/F_0}(x)$, for $a\in F_0^{\times}$. 
By definition a CM-quadratic form $q$ of gauge $a$ is isometric to $\tr_{F_0/\Qp} \circ\, \tilde{q}_a$. 
Now, notice that all quadratic forms $\tilde{q}_a$ have the same discriminant in $F_0^\times/ (F_0^\times)^2$, 
independent of $a$, as $\tilde{q}_a(x)=a\tilde{q}_1(x)$ and the quadratic space has dimension $2$. 
(Actually this discriminant is the class of the opposite of the discriminant of the extention $F/F_0$.) By  \cite[Lemma 2.2]{MR249519} this implies   \eqref{thm milnor:x}. 

For \eqref{thm milnor:2} we argue as above and consider the quadratic forms  $\tilde{q}_a$. We claim that $\tilde{q}_a\cong \tilde{q}_b$ if and only if $a$ and $b$ are in the same class of $  F^{\times}/N_{F/F_0}(F^{\times})$. This will give \eqref{thm milnor:2}  thank to \cite[Theorem 2.3]{MR249519}. The claim is elementary: if $a$ and $b$ are in the same class, we can write $a/b=N_{F/F_0}(\cf)$ and the change of variable $x\mapsto \cf \cdot  x$ gives the isometry. Conversely, if the quadratic forms  are isomorphic then the images of $F^{\times}$ by them must be the same, but those images are precisely $a \cdot N_{F/F_0}(F^{\times})$ and  $b \cdot N_{F/F_0}(F^{\times})$. 
Point \eqref{thm milnor:3} is the combination of \eqref{thm milnor:1} and \eqref{thm milnor:2}.
\end{dem}

\begin{prop}\label{prop decomp}
Let $V$ be a CM-space and  consider 
 $V\otimes_{\Qp} F$. It is endowed with two different   actions of $F$, one on the left, induced by the action of $F$ on $V$, and one on the right. To distinguish them we will write, for all $w \in V\otimes_{\Qp} F$ and $\cf\in F$, $\cf(w)$ for the left multiplication and $w\cdot \cf$ for the right multiplication. 
 The inclusion $F \subset  \tilde{F}$ gives a fixed inclusion  $V\otimes_{\Qp} F \subset V\otimes_{\Qp} \tilde{F}$ and, for all $\sigma \in \Gamma$, let us denote again by $\sigma$ the induced embedding $V\otimes_{\Qp} F \subset V\otimes_{\Qp} \tilde{F}.$   Then the following hold.
 \begin{enumerate}
\item \label{prop decomp_1}
There exists a nonzero (eigen)vector $v\in   V\otimes_{\Qp} F$ such that $\cf(v)=v\cdot\cf$. Such a $v$ is unique up to (left or right) multiplication by an element of $F^{\times}$. 
\item \label{prop decomp_2}
The (eigen)vectors $\{\sigma(v)\}_{\sigma \in \Gamma}$ form a basis of $V\otimes_{\Qp} \tilde{F}$ over $\tilde{F}$.
We set $L_{\sigma}\coloneqq\sigma(v)\cdot \tilde{F}$, it is an eigenline with respect to the embedding $\sigma$, {\itshape i.e.} $z(\sigma(v))=\sigma(v)\cdot\sigma(z)$.
\end{enumerate}
\end{prop}
\begin{dem}
Follows immediately from the classical description $F\otimes_{\Qp}\tilde{F}=\tilde{F}^{\Gamma}$ and the fact that the dimension of $V$ over $F$ is one. 
\end{dem}

\begin{prop}\label{prop decomp2}
Let $(V,q)$ be a CM-quadratic space and  let 
 $v \in V\otimes_{\Qp} F$ be the eigenvector  constructed in Proposition $\ref{prop decomp}$ above. Let  $b$ be the  bilinear pairing induced by $q$. Then the following holds.
 \begin{enumerate}
 \item \label{prop decomp2_1} One has $b(\sigma(v),\tau(v))=0$ except if $\sigma=\tau^*$.
 \item \label{prop decomp2_2} The number  $2b( v,v^*)\in F_0^{\times}$ is the gauge of $q$.
 \end{enumerate}
\end{prop}
\begin{dem}
For all $\cf\in F$ we have  $b(\sigma(\cf(v)),\tau(v))=  b(\sigma(v),\tau(\cf^*(v)))$ by definition of CM-quadratic forms. On the other hand we have the relations
$b(\sigma(\cf(v)),\tau(v))
= b(\sigma(v),\tau(v))\cdot \sigma(\cf)$ and 
$b(\sigma(v),\tau(\cf^*(v))) 
=  b(\sigma(v),\tau(v)) \cdot \tau^*(\cf)$,  by construction of $v$. These relations together give (1).

For the second point, choose $\{\sigma(v)\}_{\sigma \in \Gamma}$ as basis. A vector $w\in V\otimes_{\Qp} \tilde{F}$ is actually in $V$ if and only if there exists a scalar $\cf\in F$ such that the coordinates of $w$ with respect to this basis are $\{\sigma(\cf)\}_{\sigma \in \Gamma}$.
Let us  now  compute the quadratic form on such a vector $w\in V$:
\[q(w)=\sum_{\sigma, \tau \in \Gamma} b(\sigma(\cf v), \tau(\cf v))=\sum_{\sigma \in \Gamma} b(\sigma(\cf v), \sigma^*(\cf v))= \sum_{\sigma \in \Gamma_0} 2 \sigma(\cf \cf^*)  \sigma(b( v,v^*)), \]
where  the second equality comes from part \eqref{prop decomp2_1} and $\Gamma_0$ denotes the set of $\Qp$-embeddings of $F_0$ in $\tilde{F}$. 
In conclusion we have
 \[q(w)= \tr_{F_0/\Qp}(2b( v,v^*) \cdot N_{F/F_0}(\cf))\]
which computes the gauge.
\end{dem}
\begin{prop}\label{comp decomp}
Let $V_1 $ and $V_2$ be two CM-spaces and let $B$ be an $F$-algebra.
Suppose that it is given an identification of $B$-modules
\[V_1\otimes_{\Qp} B =  V_2\otimes_{\Qp} B \]
which is compatible with the $F$-actions.
Let $v_1$ and $v_2$ be the vectors constructed in Proposition $\ref{prop decomp}$ associated respectively with $V_1$ and $V_2$.
Then,   there exists a unique invertible element $\lambda  \in B^{\times}$ such that
\[v_1=\lambda  v_2.\]

Let now $\tilde{B}$ be an $ \tilde{F}$-algebra and suppose that   each $\sigma \in \Gamma$ extends to a  homomorphism $\tilde{\sigma} : B \rightarrow \tilde{B}$.
Then, for all $\sigma$, we have
\[\sigma(v_1)=\tilde{\sigma}(\lambda) \sigma(v_2).\]
 \end{prop}
\begin{dem}
Consider the $B$-module $V_1\otimes_{\Qp} B =  V_2\otimes_{\Qp} B $ and consider on it the left action of $F$.
By construction, $v_1$ and $v_2$ generate the same eigenline with respect to this action, this gives the existence of $\lambda$.

For the second part one applies $\tilde{\sigma}$ to the equality $v_1=\lambda  v_2$. As the $v_i$ live in  $V_i\otimes_{\Qp} F\subset  V_i\otimes_{\Qp} B  $ the action on them coincides with $\sigma.$
\end{dem}
\begin{prop}\label{burrata}
Let $(V_1,q_1)$ and $(V_2,q_2)$ be two CM-quadratic forms and let $B$ be an $F$-algebra.
Suppose that it is given an identification of $B$-modules
\[V_1\otimes_{\Qp} B =  V_2\otimes_{\Qp} B \]
which is compatible with the $F$-actions and with the quadratic forms.

Let   $\lambda \in B^{\times}$ be the scalar constructed in Proposition $\ref{comp decomp}$ and suppose that
the involution $*$ of $F$  extends to an  endomorphism $*$  of $B$ as  $\Qp$-algebra.
Then $\lambda  \cdot  *(\lambda)$ belongs to $F_0^{\times}$ and the two quadratic forms $q_1$ and $q_2$ are isomorphic if and only if 
\begin{equation}\label{eq:lambda_norma} \lambda  \cdot  *(\lambda) \in N_{F/F_0}(F^{\times}). \end{equation}
 \end{prop}
 \begin{remark}
 	Notice that  $\lambda$ is well-determined up to a constant in $F^{\times}$ since 
 	so are  $v_1$ and $v_2$, {\itshape cf.} Proposition \ref{prop decomp}. Nevertheless the condition \eqref{eq:lambda_norma} is independent of such a choice. 
 	\end{remark}
\begin{dem}
Let $v_i$ be the vectors from  Proposition \ref{comp decomp}. By Theorem \ref{thm milnor}, the two quadratic forms are isomorphic if and only if they have the same gauge in $F_0^{\times}/N_{F/F_0}(F^{\times})$. By Proposition \ref{prop decomp2}(2), this means that they are isomorphic if and only if the ratio 
$2b( v_1,v_1^*)/2b( v_2,v_2^*)$, which belongs to $F_0^{\times}$, is a norm of an element in $F$.
On the other hand, by Proposition \ref{comp decomp}, this ratio is equal  to $\lambda  \cdot  *(\lambda)$. 
\end{dem}

\section{Reduction to tensor generators}\label{sect:Rtg}

We keep the same notation as in the previous section, in particular,  $F$   is a finite extension of $\Qp$ endowed with a non trivial involution $*\colon F\rightarrow F$ and an embedding $\iota$ in a fixed Galois closure $\tilde{F}\supset F$. Moreover $F_0\subset F$ is the subfield fixed by $*$ and $\Gamma$ is the set of $\Qp$-embeddings of $F$ in $\tilde{F}$. The maximal subfield of $\tilde{F}$ which is unramified over $\Qp$ will be denoted by $\tilde{F_a}$, see Convention \ref{sb:Conv_p-adic_fields}.

\begin{definition}\label{def CM module}
 A filtered-CM-space (with respect to $F$) is the datum of a   CM-space $V$ (Definition \ref{def CM quad form}) together with an integer $n_\sigma$ associated with each eigenline $L_\sigma\subset V\otimes_{\Qp} \tilde{F}$ of Proposition \ref{prop decomp}\eqref{prop decomp_2}.
 Such a module is called symmetric if
 $n_{\sigma}=-n_{\sigma^*}.$
\end{definition}

\begin{lemma}\label{equiv CM module}
The datum of a filtered-CM-space $V$ is equivalent to the datum of a filtered $\varphi$-module $D$ over $\tilde{F}$ such that:
\begin{enumerate}
	\item $\dim_{\Fatilde}(D)=[F:\Qp]$;
	\item $D$ is endowed with an action of $F$;
	\item $D$ has a basis of vectors fixed by the Frobenius $\varphi$.
\end{enumerate}
 The equivalence goes as follows. To $V$ one associates the module $D\coloneqq V\otimes_{\Qp}\tilde{F_a}$ endowed with the Frobenius which is the identity on $V$ and extends semi-linearly. The filtration $\Fil$ on $D_{\tilde{F}}\coloneqq V\otimes_{\Qp}\tilde{F}$ is the sum of the eigenlines $L_\sigma$ such that $n_\sigma \geq i.$ Conversely to $D$ one associates $V=D^{\varphi = \id}$ and $n_\sigma=v{\dR}(L_\sigma)$, see Convention \ref{sb:Conv_Fil} for the notation $v{\dR}$. 
\end{lemma}
\begin{dem}
By definition of filtered-$\varphi$-module  each  $\Fil$ is stable under the action of $F$, hence it is the sum of eigenlines. The equivalence follows.
\end{dem}
\begin{lemma}\label{CM tensor structure}
We define the tensor product of two 
filtered-CM-spaces $(V,n_\sigma)$ and $(V',n'_\sigma)$  as   $$(V \otimes_F V' ,n_\sigma + n'_\sigma).$$ 
Under the equivalence of Lemma $\ref{equiv CM module}$ this corresponds to taking the two corresponding  filtered-$\varphi$-modules $D$ and $D'$ and to considering inside the  filtered-$\varphi$-module $D\otimes D'$ the sub-filtered-$\varphi$-module where the two $F$-actions coincide.
\end{lemma}
\begin{dem}
Let $\{L_\sigma\}_{\sigma \in \Gamma}$ and $\{L'_{\sigma}\}_{\sigma \in \Gamma}$ be the eigenlines associated with the actions of $F$ on $D$ and $D'$. Then the   action of $F\otimes_{\Qp}F$ on  $D\otimes D'$ has $\{L_\sigma \otimes L'_{\tau}\}_{\sigma,\tau \in \Gamma}$ as eigenlines. By definition of the filtration on a tensor product one has $v{\dR} (L_\sigma \otimes L'_{\tau}) = n_\sigma + n'_\tau$. On the other hand the submodule of $D\otimes D'$   where the two $F$-actions coincide has $L_\sigma \otimes L'_{\sigma}$ as eigenline corresponding to $\sigma \in \Gamma$, hence we do have that $n_\sigma + n'_\sigma$ is its de Rham valuation.
\end{dem}
\begin{definition}\label{fundamental module}
The fundamental filtered-CM-space (with respect to $(F,\iota)$) is the (symmetric) filtered-CM-space $F(\iota)$ whose underling CM-space is $F$ 
and  whose associated integers are 
 \[n_{\iota}=1, \hspace{0.5cm} n_{\iota^*}=-1,  \hspace{0.5cm}  \textrm{and} \hspace{0.5cm} n_{\sigma}=0,  \hspace{0.5cm} \forall \sigma\neq \iota,\iota^* .\]
\end{definition}
\begin{remark} \label{rem:tensor_gen}
By construction, the set of fundamental filtered-CM-spaces $\{F(\tau)\}_{\tau \in \Gamma}$ is a family of tensor generators of the category of symmetric filtered-CM-spaces. 
\end{remark}

\begin{prop}\label{prop admissibility}
If a filtered-CM-space is symmetric (Definition $\ref{def CM module}$) then its corresponding  filtered-$\varphi$-module is admissible.
\end{prop}
\begin{dem}
Admissibility is stable under tensor product and direct factor, hence it is enough to study the set $\{F(\tau)\}_{\tau \in \Gamma}$
of tensor generators of the symmetric    filtered-$\varphi$-modules, {\itshape cf.} Remark \ref{rem:tensor_gen}. 
Hence it is enough to show admissibility for the fundamental filtered-CM-space $F(\iota)$ (for all pairs $(F,\iota)$). 
Let $D$ be the  filtered-$\varphi$-module corresponding to $F(\iota)$ by Lemma \ref{equiv CM module}. 
By \cite{MR1779803} it is enough to check the condition of weakly-admissibility for $D$.
The Newton polygon is constantly zero by construction. Hence the only sub-filtered-$\varphi$-modules which might contradict the condition of weakly-admissibility are those containing the eigenline $L$ of de Rham valuation one. Let $N$ be such a module and let us show that $N=D$, which will give the admissibility.
%
%
 It is not restrictive to assume that $N$ is stable through the action of $F$, by \cite{MR1944572}*{Proposition~3.1.1.5} (or by a short direct argument).
Then  the inclusion $N^{\varphi = \id} \subset D^{\varphi = \id}$ is $F$-equivariant, hence $N^{\varphi = \id}=0$ or $N^{\varphi = \id} =  D^{\varphi = \id}$. 
On the other hand,  as $D$ is generated by its $\varphi$-invariant part, the Newton slopes of  $N$ are also zero, {\itshape cf.} \cite[\S(1.3) and Lemma 1.3.4]{MR563463}. Now, it is well-know that for any $\Fro$-module $Q$ we have 
$\dim_{\Qp}(Q^{\Fro = \id} )\leq \dim_{\tilde{F}_a}(Q) $. By construction $\dim_{\Qp}(D^{\Fro = \id}) = \dim_{\tilde{F}_a} (D)$ and $N$ is not zero (it contains $L$); thus by arguing on the dimension over $\Qp$ in the exact sequence 
$$ 0\rightarrow N^{\varphi = \id} \rightarrow D^{\Fro = \id} \rightarrow (D/N)^{\Fro = \id},$$
we get $N^{\varphi = \id}= D^{\Fro = \id}$. Therefore $N=D$.
\end{dem}
\begin{definition}\label{fundamental period}
Let $V$ be a symmetric filtered-CM-space, $D$ be the corresponding  admissible filtered $\varphi$-module 
({\itshape cf.} Lemma~\ref{equiv CM module} and Proposition \ref{prop admissibility}). Let  $W$ be the crystalline representation of $G_{\tilde{F}}$ corresponding to $D$ via Fontaine's $p$-adic comparison theorem. Consider the identification 
\[  W \otimes_{\Qp} \Bcris  = D \otimes_{\Fatilde} \Bcris = V \otimes_{\Qp} \Bcris, \]
and the induced one by tensoring it with $F$ over $\Fa$,
\[V\otimes_{\Qp} \BcrisF =   W \otimes_{\Qp} \BcrisF,\]
where \BcrisF\ is defined in Convention \ref{sb:Conv_pHT}\eqref{BcrisE1}.
Then we define the period associated with $V$ as the invertible element $\lambda_V \in \BcrisF^{\times}$ deduced from Proposition 
\ref{comp decomp} applied to 
  $V_1=V , V_2=W$ and $B = \BcrisF$. We recall that $\lambda_V$ is well-determined up to a constant in $F^{\times}$ since 
	it is defined as the period such that $$v_1 = \lambda_V v_2,$$ for a given choice of a  basis $v_1$ (resp. $v_2$) of the $F$-line in 
	$V \otimes_{\Qp} F$ (resp. $W \otimes_{\Qp} F$) where the two $F$-actions coincide, see Proposition \ref{prop decomp}. 
 
 We define \emph{the fundamental period} (with respect to $(F,\iota)$) 
 \begin{equation} \label{eq:the_fund_period}
 	\lambda=\lambda_{F(\iota)}
 	\end{equation}  
 	as the period associated with the  fundamental  filtered-CM-space $F(\iota)$ of Definition \ref{fundamental module}.
 \end{definition}

The goal of the next two sections is to establish the following theorem, whose proof is finally given in \ref{proof:not_a_norm}.
 
 \begin{theorem}\label{th:not_a_norm}
 Let $\lambda\in\BcrisF$ be the fundamental period defined in Definition $\ref{fundamental period}$ above. 
 For every endomorphism $*\colon\BcrisF\rightarrow\BcrisF$ extending the involution $*$ of $F$, we have  
 \begin{equation} \label{eq:not_a_norm}
 \lambda \cdot {*}(\lambda) \notin N_{F/F_0}(F^{\times}).
 \end{equation}
 \end{theorem}
\begin{remark}\label{rem:not_a_notm}
	\begin{enumerate}
		\item Notice that we have $\lambda \cdot {*}(\lambda) \in F_0^{\times}$ by applying Proposition \ref{burrata}, with $B$ equals to $\BcrisF$.
		\item \label{rem:not_a_notm2}
		 It is enough to prove \eqref{eq:not_a_norm} for one extension $*\colon \BcrisF\rightarrow\BcrisF$, since by 
		Proposition \ref{burrata}, the condition \eqref{eq:not_a_norm} is equivalent to a statement not involving the choice of the extension.
		\item In \ref{sb:BcrisF_more} we will construct such an extension $*=*\cris$ 
		and we will work with it through the next two sections.
	\end{enumerate}
\end{remark}
%
%
 \begin{cor}\label{cime di rapa}
  Let $V$ be a symmetric filtered-CM-space, $\lambda_V$ be the corresponding period (Definition $\ref{fundamental period}$) and $*$ as in Theorem \ref{th:not_a_norm}. 
  Then 
 $\lambda_V  \cdot{*}(\lambda_V) \in F_0^{\times}$ belongs to $N_{F/F_0}(F^{\times})$ if and only if the non negative  integer 
$$\sum\limits_{\sigma, n_\sigma \geq 0} n_\sigma$$
 is even, where the sum is taken over all $\sigma\in\Gamma$ such that $n_\sigma \geq 0$.
 \end{cor}
 
 \begin{dem}
Recall that the group $F_0^{\times}/N_{F/F_0}(F^{\times})$ has cardinality two (Theorem \ref{thm milnor}). By construction 
the map $V\mapsto \lambda_V$ from symmetric filtered-CM-spaces to $\BcrisF^{\times}$ is multiplicative on tensor products.
In particular the statement is stable under tensor product, hence it is enough to check it on tensor generators. By Remark \ref{rem:tensor_gen} it is enough to show  the statement for the fundamental  filtered-CM-space $F(\iota)$, which is Theorem \ref{th:not_a_norm} above.
\end{dem}

\begin{remark}\label{cozze}
By construction, the integer $-\sum_{\sigma, n_\sigma \geq 0} n_\sigma$   is the 
 minimum  of the Hodge polygon of the filtered $\Fro$-module associated with the symmetric filtered-CM-space $V$ of Corollary \ref{cime di rapa} above.
\end{remark}

\section{Characterization of $p$-adic periods}\label{sect:p-adic_periods}

In $\eqref{eq:the_fund_period}$ we defined a period $\lambda$ in $\BcrisF$, which we called the fundamental period. 
The goal of this section is to give some properties of $\lambda$ as element of $\BcrisF$ which are enough to 
characterize it up to a constant in $F^{\times}$, see  Propositions \ref{prop:4prop_lambda} and \ref{prop:unicity}.

We keep the same notation as in Sections \ref{sect:CQ-qf} and \ref{sect:Rtg}, in particular $F$ is a finite extension of \Qp\ with  ramification index $e$ and residual degree $f$. 
Recall that $F$ is endowed with an involution $*\colon F\rightarrow F$, whose  subfield of fixed points  is denoted by $F_0$, which has not to be confused with the maximal unramified subfield of $F$, denoted by $\Fa$.
The degree of $F$ over $F_0$ is $2$, 
so $ef = 2[F_0:\Qp]$.

We start this section with some preliminary constructions and lemmas, certainly known to experts, that we recall here for the convenience of the reader. 
For the notation on the rings of $p$-adic periods we refer to Convention \ref{sb:Conv_pHT}.


\subsection{}\label{sb:def_nu} 
In Sections \ref{sect:CQ-qf} and \ref{sect:Rtg} we have fixed a Galois closure $\tilde{F}$ of $F$ and 
an  embedding $\iota\colon F\hookrightarrow \tilde{F}$; we have then denoted by $\Gamma$ the set 
of \Qp-embeddings of $F$ in $\tilde{F}$ and  by $*$ the embedding $\iota\circ * \in \Gamma$.
From now on, let us  consider $\tilde{F}$ as a subfield of $\Bdr$ by choosing an embedding of $\tilde{F}$ in $\Bdr$. We can then identify  $\Gamma$ to the  set  
of \Qp-embeddings of $F$ in $\Bdr$. 
We will also identify $F$ to its image $\iota(F)\subset\tilde{F}\subset\Bdr$.  
Since $\Fa/\Qp$ is a cyclic unramified extension, every $\sigma$ in $\Gamma$ stabilizes $\Fa$ and its restriction to $\Fa$ is a power of the absolute Frobenius $\Fro$ of $\Fa$: we will set
 $\sigma_{|\Fa} = \Fro^{\nu(\sigma)}$, for a unique integer  $0\leq \nu(\sigma)\leq f-1$. By construction, we have $\nu(\iota) =0$.
 We put $\nu\coloneqq\nu(*)$, it is either $0$ or $f/2$ whether $F/F_0$ is ramified or unramified.

\subsection{}\label{sb:BcrisF_more} 
We keep the notation of \ref{sb:def_nu}.
We denote by \BcrisF\ the smallest subring of \Bdr\ containing \Bcris\ and $F$, {\itshape cf.} 
Convention~\ref{sb:Conv_pHT}\eqref{BcrisE1}; the ring \BcrisF\ is 
 identified to the image of the natural map $\Bcris\otimes_{\Fa}F\rightarrow\Bdr$.
Similarly, we will use $\BcrisFtilde$, which is isomorphic to $\Bcris\otimes_{\tilde{F}_a}\tilde{F}$, 
where $\tilde{F}_a$ is the absolute unramified subfield of $\tilde{F}$. 

We recall  that the ring 
\BcrisF\ is endowed with an $F$-linear endomorphism $\Fro^f\colon\BcrisF\rightarrow\BcrisF$, 
defined as $ \Fro[\mathrm{cris}]^f \otimes_{\id[\Fa]} \id[F]$, where $\Fro[\mathrm{cris}]$ is the Frobenius of \Bcris, {\itshape cf.} Convention \ref{sb:Conv_pHT}\eqref{BcrisE3}.

Furthermore, every $\sigma$ in $\Gamma$ extends to an homomorphism,  denoted by $\sigma\cris\colon \BcrisF \rightarrow \BcrisFtilde$, defined as 
\begin{equation}
\sigma\cris\colon\BcrisF=\Bcris\otimes_{\Fa}F \xrightarrow{\Fro[\mathrm{cris}]^{\nu(\sigma)} \otimes \sigma}
\Bcris\otimes_{\tilde{F}_a}\tilde{F}=\BcrisFtilde,
\label{eq:sigma_on_BcrisF}
\end{equation}
where the product $\Fro[\mathrm{cris}]^{\nu(\sigma)} \otimes \sigma$ 
is taken over the inclusion $\Fa \hookrightarrow {\tilde{F}_a}$, 
{\itshape cf.} \cite[A.II, \S3, n\textsuperscript{o}3, p.53]{MR0274237}.
Note that  $\iota\cris$ coincides with the inclusion $\BcrisF\subset \BcrisFtilde$ as subrings of $\Bdr$.
Moreover, although we will not need it in following, it is not difficult to show that $\sigma\cris$ is injective,
by using the fact that the natural maps $\Bcris\otimes_{\Fa}F\rightarrow \Bdr$, $\Bcris\otimes_{\tilde{F}_a}\tilde{F}\rightarrow \Bdr$ and the Frobenius $\Fro[\mathrm{cris}]$ are.

 Finally, the involution $*$ of $F$  extends also to an endomorphism  $*\cris\colon\BcrisF\rightarrow\BcrisF$ 
by putting $*\cris\coloneqq\Fro[\mathrm{cris}]^{\nu}\otimes_{\Fro^{\nu}} *$. Beware that in general it is not an involution of \BcrisF: precisely, as it follows from Lemma \ref{lemma:Gamma_cris} below,
if $F/F_0$ is unramified, then $*\cris^2 = \Fro[\mathrm{cris}]^f$; if $F/F_0$ is ramified, then $*\cris^2 = \id[\BcrisF]$.  

\begin{lemma}\label{lemma:Gamma_cris}
 We keep the notation of $\ref{sb:def_nu}$ and $\ref{sb:BcrisF_more}$.
\begin{enumerate}
	\item \label{lemma:Gamma_cris:1}
	For all $\sigma$ in $\Gamma$,
we have $\sigma\cris\circ *\cris = (\sigma\circ *)\cris$ except if $F/F_0$ is unramified and $\nu(\sigma)\geq f/2$, as in that case $\sigma\cris\circ *\cris = (\sigma\circ *)\cris \circ \Fro[\mathrm{cris}]^f$.																			
\item \label{lemma:Gamma_cris:2}
$*\cris\circ \Fro[\mathrm{cris}]^f =
 \Fro[\mathrm{cris}]^f \circ *\cris.$ 
\end{enumerate}
\end{lemma}
\begin{dem} 
The property \eqref{lemma:Gamma_cris:2} follows directly by definition:
$$*\cris\circ \Fro[\mathrm{cris}]^f = 
(\Fro[\mathrm{cris}]^{\nu}\otimes *) \circ (\Fro[\mathrm{cris}]^f \otimes \id[F])= 
\Fro[\mathrm{cris}]^{\nu+f}\otimes * =
(\Fro[\mathrm{cris}]^f \otimes \id[F]) \circ (\Fro[\mathrm{cris}]^{\nu}\otimes *) =
 \Fro[\mathrm{cris}]^f \circ *\cris. $$

Let us prove \eqref{lemma:Gamma_cris:1}. 
We have $$\sigma\cris\circ *\cris = (\Fro[\mathrm{cris}]^{\nu(\sigma)}\otimes_{\Fro^{\nu(\sigma)}} \sigma)\circ(\Fro[\mathrm{cris}]^{\nu}\otimes_{\Fro^{\nu}} *)= \Fro[\mathrm{cris}]^{\nu(\sigma)+\nu}\otimes_{\Fro^{\nu(\sigma)+\nu}} (\sigma\circ *).$$
If $F/F_0$ is ramified then $\nu=0$ and $\nu(\sigma\circ *)=\nu(\sigma)$, \emph{cf.}~\ref{sb:def_nu}, and the statement is clear. 
If $F/F_0$ is unramified then $\nu=f/2$: if $\nu(\sigma)<f/2$, the statement is also clear. 
Finally, if $F/F_0$ is unramified and $\nu(\sigma)\geq f/2$, we have $\nu(\sigma)+f/2=f+\nu(\sigma\circ*)$, so   
\begin{equation*}
\Fro[\mathrm{cris}]^{\nu(\sigma)+{f}/{2}}\otimes_{\Fro^{\nu(\sigma)+{f}/{2}}} (\sigma\circ *)=
\Fro[\mathrm{cris}]^{\nu(\sigma\circ*)+f}\otimes_{\Fro^{\nu(\sigma\circ*)+f}} (\sigma\circ *)=
(\sigma\circ *)\cris\circ\Fro[\mathrm{cris}]^{f}.\qedhere
\end{equation*}
\end{dem}

\begin{conv} \label{conv:sigma_cris}
When there is no risk of confusion we will write abusively $\sigma$ (resp. $*$, resp. $\Fro$) instead of $\sigma\cris$ (resp. $*\cris$,
resp. $\Fro[\mathrm{cris}]$);
also, for any $\lambda\in\BcrisF$,
we will sometimes write  $\lambda^*$ (resp. $\lambda^{*\cris}$) instead of $*(\lambda)$ (resp. ${*\cris}(\lambda)$). 
\end{conv}

\begin{lemma}[Colmez]\label{lemma:discesa_F}
 We keep the notation of $\ref{sb:def_nu}$, $\ref{sb:BcrisF_more}$ and Convention $\ref{conv:sigma_cris}$. 
Let $\mu\in \BcrisF$ be an element such that:
\begin{enumerate} 
	\item\label{lemma:discesa_F:1} for all $\sigma\in\Gamma$, $v\dR(\sigma(\mu))=0$;
	\item\label{lemma:discesa_F:2} $\Fro^f(\mu)=\mu$.
\end{enumerate} 
Then $\mu$ belongs to $F^{\times}\subset\BcrisF$. 
\end{lemma}
\begin{dem} 
This follows from the fundamental exact sequence  \cite{MR1956055}*{Lemma 9.25 (SEF3E)}. 
We briefly recall the argument.   
By construction $\Bcris$ is the localization $\mathrm{A}\cris[{t}^{-1}]$ of Fontaine's ring $\mathrm{A}\cris$, 
where  
$t$ is the period of $\Qp(1)$, see Convention \ref{sb:Conv_pHT}\eqref{Bcris_3}.
 Thus 
there exists some integer $n\geq 0$, such that $\mu t^n \in \BcrisF^{+}\coloneqq\mathrm{A}\cris\otimes_{W(k_F)}F$.
By evaluating the sequence of \cite[Lemma~9.16]{MR1956055} in $\Cp$ and by multiplying with $t^{n-1}$ we get a short exact sequence 
of $\Qp$-vector spaces
\begin{equation}
\label{eq:seFondamentale}
0\rightarrow F\cdot t^{n} \rightarrow (B_F)^{\Fro^{f}=p^{nf}}\xrightarrow{\,\Theta_{F}\,} \bigoplus_{\Gamma}\Cp\rightarrow 0,
\end{equation}
where $B_F
\subset \Bdr^+$ is some ring of periods containing  $\BcrisF^{+}$, {\itshape cf.} \cite{MR1956055}*{\S8.5}\footnote{The ring $B_F$ is the ring $\mathrm{B}_{\mathrm{max},F}^+$ in the notation of {\itshape loc.cit.}};
 and the map $\Theta_{F}$ is given by 
$$\Theta_{F}\colon x 
\longmapsto 
\Bigl(\theta\bigl(\sigma\bigl({x}{t^{-n+1}}\bigr)\bigr)\Bigr)_{\sigma\in\Gamma}\,,$$ where 
 $\theta\colon\Bdr^+\rightarrow\Cp$ denotes the reduction map and $\sigma\colon 
B_F\rightarrow \Bdr^+$
 extends the homomorphism $\sigma\cris$ of \eqref{eq:sigma_on_BcrisF}.
By the hypothesis \eqref{lemma:discesa_F:2} of the statement, the element $x=\mu t^n$ belongs to $ (\BcrisF^+)^{\Fro^{f}=p^{nf}}$;
we get $\Theta_F(x)= (\theta(\sigma(\mu t)))_{\sigma\in\Gamma}$. 
By the hypothesis \eqref{lemma:discesa_F:1} of  the statement we have $\vdR(\sigma(\mu))=0$ and since 
 $\sigma{\cris}(t)= \Fro[\mathrm{cris}]^{\nu(\sigma)}(t)= p^{\nu(\sigma)}t$,
we get $\Theta_F(x)=0$.
 Hence $\mu\in F$;  
moreover, $\mu$ is nonzero otherwise the hypothesis \eqref{lemma:discesa_F:1} above would fail.
\end{dem}

\begin{definition} \label{def:abstract_fundamental_period}
An element $\beta$ in $\BcrisF$ is called a \emph{fundamental period} if it
satisfies the following  properties:
\begin{enumerate}[a)] 
\item\label{prop:4prop_lambda_1} $\vdR(\beta)=1$;
\item\label{prop:4prop_lambda_2} $\vdR(\beta^*)=-1$;
\item\label{prop:4prop_lambda_3} for all $\sigma\in\Gamma\backslash\{\iota,*\}$, $\vdR(\sigma(\beta))=0$;
\item\label{prop:4prop_lambda_4} $\Fro^f(\beta)=\beta$;
\end{enumerate}
where $\sigma(\beta)$, $\Fro^f$, 
and $\beta^*
\in\BcrisFtilde\subset \Bdr$ are defined in $\ref{sb:BcrisF_more}$ and Convention $\ref{conv:sigma_cris}$. 

We denote by $\mathcal{P}\subset \BcrisF$ the set of fundamental periods.
\end{definition}

\begin{prop}\label{prop:4prop_lambda} 
The fundamental period $\lambda\in\BcrisF^{\times}$ defined in \eqref{eq:the_fund_period}
is a fundamental period according to Definition $\ref{def:abstract_fundamental_period}$.
\end{prop}
\begin{dem} 
 In this proof let us write  $V$ for  the fundamental filtered-CM-space $F(\iota)$, {\itshape cf.} Definition $\ref{fundamental period}$.
 Let $D\coloneqq V\otimes_{\Qp}\Fatilde$ be the admissible filtered-$\varphi$-module corresponding to $F(\iota)$ by Lemma \ref{equiv CM module} and 
 Proposition \ref{prop admissibility}.
Let $W$ be the Galois representation associated with 
$D$, which comes with an identification of $\BcrisFtilde$-modules
$$ V \otimes_{\Qp} \BcrisFtilde =   W \otimes_{\Qp} \BcrisFtilde\,.$$
In this context, Proposition \ref{comp decomp} gives vectors  $v_1\in V_F\coloneqq V\otimes_{\Qp} F$
and $v_2\in W_F\coloneqq W\otimes_{\Qp} F$ and the relation  $v_1=\lambda v_2$.
By applying $\sigma\in\Gamma$, we get $\sigma(v_1)=\sigma(\lambda) \sigma(v_2)$.
Now, by construction we know that:
\begin{itemize}
	\item[-] for all $\sigma\in\Gamma$, $\vdR(\sigma(v_2))=0$; 
	\item[-] $\vdR(v_1)=1$ and $\vdR(v_1^*)=-1$;   
	\item[-] for all 
$\sigma\in\Gamma\backslash\{\iota,*\}$, $\vdR(\sigma(v_1))=0$.
\end{itemize}
 Therefore, we get:  
 \begin{itemize}
	\item[-] for $\sigma=\iota$,  $v_1=\lambda v_2$, so $\vdR(\lambda)=1$; 
	\item[-] for $\sigma=*$, $v_1^*=\lambda^* v_2^*$,  so $\vdR(\lambda^*)=-1$;
	\item[-] for $\sigma\in\Gamma\backslash\{\iota,*\}$, $\sigma(v_1)=\sigma(\lambda) \sigma(v_2)$,  so $\vdR(\sigma(\lambda))=0$.
\end{itemize}
Finally, by applying $\Fro[]^f$ to $v_1=\lambda v_2$,
since by construction both $v_1$ and $v_2$ are fixed by $\Fro[]^f$, we get 
$$ v_1=\Fro[D]^f(v_1)=\Fro[]^f(\lambda) \Fro[]^f (v_2)=\Fro[]^f(\lambda) v_2,$$
hence $\Fro^f(\lambda)=\lambda$.
\end{dem}

\begin{prop}[Uniqueness of fundamental periods]\label{prop:unicity} 
Let $\lambda\in \BcrisF^{\times}$ be as in Proposition $\ref{prop:4prop_lambda}$ and $\mathcal{P}$ the set of fundamental periods of Definition $\ref{def:abstract_fundamental_period}$, then 
$$ \mathcal{P}=F^{\times}\cdot\lambda.$$  
In particular all fundamental periods are invertible.
\end{prop}
\begin{dem} Since $\lambda$ is invertible in \BcrisF, we can consider the subset $\lambda^{-1}\mathcal{P}$ of \BcrisF. 
We have to prove $\lambda^{-1}\mathcal{P}=F^{\times}$, which is exactly the statement of Lemma \ref{lemma:discesa_F}.\end{dem}

\begin{prop}\label{cor:lambda_lambda*} 
Let $\lambda\in \BcrisF$ be a fundamental period.
Then 
$ \lambda \lambda^{*} \in F_0^{\times}.$
\end{prop}

\begin{remark} \label{rem:not_using_quad_forms}
For the fundamental period $\lambda$ defined in \eqref{eq:the_fund_period}, the statement
of Proposition \ref{cor:lambda_lambda*} follows directly from Proposition~\ref{burrata}. Thanks to Proposition~\ref{prop:unicity}  the same is true for any fundamental period.
We give below an alternative proof based on the $p$-adic properties of the period 
 which does not use quadratic forms.
\end{remark}

\begin{dem} In this proof in order to avoid any possible confusion we will write $\lambda^{*\cris}$ instead of $\lambda^{*}$.
First, let us show that the period $ \lambda \lambda^{*\cris} \in \BcrisF$ satisfies the hypothesis of Lemma \ref{lemma:discesa_F}.
For any $\sigma$ in $\Gamma$ we have 
$$ \sigma\cris(\lambda \lambda^{*\cris}) = \sigma\cris(\lambda) \cdot (\sigma\cris\circ{*\cris})(\lambda)  $$
and by Lemma \ref{lemma:Gamma_cris}\eqref{lemma:Gamma_cris:1}, $(\sigma\cris\circ{*\cris})(\lambda)$ is equal either to $(\sigma\circ{*})\cris(\lambda)$ or 
to $(\sigma\circ{*})\cris(\Fro[\mathrm{cris}]^f(\lambda))$. Since 
$\Fro[\mathrm{cris}]^f(\lambda)=\lambda$ by property \ref{prop:4prop_lambda_4} of Definition \ref{def:abstract_fundamental_period}, we get 
$$\sigma\cris(\lambda \lambda^{*\cris}) = \sigma\cris(\lambda)\cdot(\sigma\circ{*})\cris(\lambda).$$ 
Applying $\vdR$ and considering the properties  \ref{prop:4prop_lambda_1}, \ref{prop:4prop_lambda_2} and \ref{prop:4prop_lambda_3}
of Definition \ref{def:abstract_fundamental_period} for $\lambda$, we get that $\lambda \lambda^{*\cris}$ satisfies the property \eqref{lemma:discesa_F:1} of Lemma \ref{lemma:discesa_F}. Let us check  the property \eqref{lemma:discesa_F:2} of Lemma \ref{lemma:discesa_F}:
$$\Fro[\mathrm{cris}]^f(\lambda \lambda^{*\cris}) = \Fro[\mathrm{cris}]^f(\lambda) \Fro[\mathrm{cris}]^f(\lambda^{*\cris}) = \lambda \lambda^{*\cris},$$
because $\Fro[\mathrm{cris}]^f$ and $*\cris$ commute, \emph{cf.} Lemma \ref{lemma:Gamma_cris}\eqref{lemma:Gamma_cris:2}, and $\Fro[\mathrm{cris}]^f(\lambda)=\lambda$.
Therefore, by Lemma \ref{lemma:discesa_F}, the period $\lambda \lambda^{*\cris} \in \BcrisF$  belongs to $F^{\times}\subset\BcrisF$ . Applying $*$ we get 
$$  (\lambda \lambda^{*\cris})^* = \lambda^{*\cris} \lambda^{*\cris^2}= \lambda^{*\cris} \Fro[\mathrm{cris}]^{2\nu}(\lambda) = \lambda^{*\cris} \lambda,$$
where $\nu=0$ or $f/2$, \emph{cf.} \ref{sb:def_nu} and \ref{sb:BcrisF_more}; and we used again $\Fro[\mathrm{cris}]^{f}(\lambda)=\lambda$.
Being stable under $*$, the element $ \lambda \lambda^{*\cris}\in F^{\times}$ belongs actually to $F^{\times}_0$. 
\end{dem}

\section{Lubin--Tate periods and local reciprocity} \label{sect:LT_and_reciprocity}

The goal of this section is the proof of Theorem \ref{th:not_a_norm}, which is given in \ref{proof:not_a_norm}.
Before that we need to relate the fundamental period $\lambda$ of \eqref{eq:the_fund_period} to the  Lubin--Tate periods, \emph{cf.} Proposition \ref{prop:explicit_lambda} and Corollary \ref{cor:explicit_lambda}. 

 We keep the notation of Section \ref{sect:p-adic_periods} and specifically we refer to $\ref{sb:def_nu}$, $\ref{sb:BcrisF_more}$ and Convention $\ref{conv:sigma_cris}$. 
In particular, recall that $F$ is a $p$-adic field endowed with non trivial involution $*$ and we denoted by $F_0$ the subfield of $F$ of points fixed by $*$, whereas we denoted by $\Fa$ the absolute unramified subfield of $F$. 

\begin{definition} \label{def:Lubin-Tate_period}
Let $\pi$  be a uniformizer of $F$. We say that $\alpha$ in $\BcrisF$ is a Lubin--Tate period (relative to $\pi\in F$) if: 
\begin{enumerate}
	\item \label{def:Lubin-Tate_period:1} 
	$v\dR(\alpha)=1$;
	\item \label{def:Lubin-Tate_period:2}
	for all $\sigma\in \Gamma\backslash\{ \iota\}$, $v\dR(\sigma(\alpha))=0$;
	\item \label{def:Lubin-Tate_period:3}
	$\Fro^f(\alpha)=\pi\alpha$;
\item \label{def:Lubin-Tate_period:4}
$\alpha$ is invertible in $\BcrisF$.
\end{enumerate}
\end{definition}

\begin{theorem}[Colmez]
There exist Lubin--Tate periods relative to any choice of a uniformizer  $\pi$ of $F$.
\end{theorem}

\begin{remark}
\begin{enumerate}
	\item Lubin--Tate periods are constructed in \cite[\S9.3-\S9.5]{MR1956055}\footnote{Precisely, {\itshape cf. loc.cit.} 
Proposition 9.10 and Lemma 9.18. Note
that these periods are constructed in $B_{\mathrm{max},F}$, but actually they belong to $\BcrisF$, since, by Definition 
\ref{def:Lubin-Tate_period}\eqref{def:Lubin-Tate_period:3} above, they belong to $\Fro^f(B_{\mathrm{max},F})\subset \BcrisF$.},
 via a direct computation in periods rings using  Lubin--Tate formal groups \cite{MR172878}, 
see also \citelist{\cite{MR1247996}*{\S2} \cite{fourquaux_these}*{\S3.6}}. 
Therefore the action of $G_F$ on $\alpha\in\BcrisF$ is given precisely by multiplication with the Lubin--Tate character. 
We will not need this property, so we did not mention it in Definition \ref{def:Lubin-Tate_period} above.
	\item To keep this article more self-contained,  we recall in Appendix \ref{sect:LubinTate} another construction of Lubin--Tate periods, relying on the theorem 
``weakly admissible $\Rightarrow$ admissible'', 
\emph{cf.} \cite{MR1779803}, applied  to a well-chosen 
filtered $\Fro$-module over $F$, see Definition \ref{sb:def_D_pi_Filtration}.   
\end{enumerate}
\end{remark}

\subsection{}\label{sb:Fa_nr_cappuccio}
 Let us denote by $\widehat{F}\nr$ (resp. $\widehat{F}_a\nr$) the completion of maximal unramified extension of $F$ (resp. $\Fa$)
and by $\Gal{\widehat{F}\nr}{F}$ the group of continuous $F$-automorphisms of $\widehat{F}\nr$.
 We have  $\widehat{F}\nr = \widehat{F}_a\nr \otimes_{\Fa} F$, that we embed canonically in $\BcrisF$. 
By definition, the endomorphism 
$\Fro[\mathrm{cris}]^f$ of \BcrisF, restricted to $\widehat{F}\nr$, is the arithmetic Frobenius $\Fro^f$ in $\Gal{\widehat{F}\nr}{F}$,
\emph{i.e.} the unique isomorphism of $\widehat{F}\nr$ lifting the map $x\mapsto x^{q}$, $q=p^f$, on the residue field.

\begin{prop}[Explicit construction of the fundamental period]\label{prop:explicit_lambda}
Assume that $\alpha\in\BcrisF$ is a Lubin--Tate period, for some uniformizer $\pi\in F$.  Then there exists $c\in (\widehat{F}\nr)^{\times}$ such that
$c\frac{\alpha}{\alpha^*}$ is a fundamental period, see Definition $\ref{def:abstract_fundamental_period}$.
Moreover $c$ can be chosen as follow:
\begin{itemize}
	\item[-]  if $F/F_0$ is unramified, and $\pi$ belongs to $F_0$, then $c=1$; 
	\item[-] if $F/F_0$ is (totally) ramified, then $c$ satisfies
\begin{equation} 
\label{eq:Cocycle_wild}
\Fro^f(c) = \frac{\pi^*}{\pi} c,
\end{equation}   
where $\Fro^f\in\Gal{\widehat{F}\nr}{F}$ is the arithmetic Frobenius, see $\ref{sb:Fa_nr_cappuccio}$ above.
	\end{itemize}
\end{prop}
\begin{dem}
The Lubin--Tate period $\alpha$ is invertible in $\BcrisF$, so is ${*\cris}(\alpha)$. Consider the period 
$$\frac{\alpha}{\alpha^*} \coloneqq {\alpha}\cdot{{*\cris}(\alpha)}^{-1} \in \BcrisF.$$
First let us show that $\frac{\alpha}{\alpha^*}$ 
satisfies the properties \ref{prop:4prop_lambda_1}, \ref{prop:4prop_lambda_2} and
\ref{prop:4prop_lambda_3} of Definition \ref{def:abstract_fundamental_period}; then, we will multiply it by a well chosen constant $c$ in $\widehat{F}\nr$ in order to get the property \ref{prop:4prop_lambda_4} of \ref{def:abstract_fundamental_period}. 
Since, for every $\sigma$ in $\Gamma$, $\sigma\cris(\widehat{F}\nr)=\widehat{F}\nr\subset \BcrisFtilde\cap \Bdr^+$, normalizing the period by a nonzero constant 
in $\widehat{F}\nr$ will not change the first three properties.  

As $\alpha$ is a Lubin--Tate period, by using  Definition \ref{def:Lubin-Tate_period}, we get  the following.
\begin{enumerate}[a)]
\item 
$$\vdR\bigl(\tfrac{\alpha}{\alpha^*}\bigr)= \vdR(\alpha) - \vdR(*\cris({\alpha})) = 1-0=1.$$
\item
$$\vdR\bigl(*\cris\bigl(\tfrac{\alpha}{\alpha^*}\bigr)\bigr)= \vdR(*\cris({\alpha})) - \vdR(*\cris^2({\alpha})) 
=- \vdR(*\cris^2({\alpha})).$$
Since $*\cris^2$ is either $\id[\BcrisF]$ or $\Fro[\mathrm{cris}]^f$, {\itshape cf.} \ref{sb:BcrisF_more}, and $\Fro[\mathrm{cris}]^f(\alpha)=\pi\alpha$, we get
 $- \vdR(*\cris^2({\alpha}))=-1$.
\item For every $\sigma$  in $\Gamma\backslash\{\iota,*\}$,
$$ \vdR\bigl(\sigma\cris\bigl(\tfrac{\alpha}{\alpha^*}\bigr)\bigr) = \vdR(\sigma\cris({\alpha})) - \vdR(\sigma\cris(*\cris({\alpha})))=
- \vdR(\sigma\cris(*\cris({\alpha}))),$$
as $\sigma\not = \iota$. 
By using Lemma \ref{lemma:Gamma_cris}\eqref{lemma:Gamma_cris:1}, 
the period $\sigma\cris(*\cris({\alpha}))$ is either 
$(\sigma\circ *)\cris({\alpha})$ or $(\sigma\circ *)\cris(\Fro[\mathrm{cris}]^f({\alpha}))=(\sigma\circ *)\cris(\pi\alpha)=
\sigma(\pi^*)(\sigma\circ *)\cris(\alpha)$.  As $\sigma(\pi^*)\in \tilde{F}\subset\BcrisFtilde\cap\Bdr^+$ and  $\sigma\circ * \not = \iota$, 
 we have, in both cases,
 $\vdR(\sigma\cris(*\cris({\alpha})))=0$ by Definition \ref{def:Lubin-Tate_period}\eqref{def:Lubin-Tate_period:3}. 
\end{enumerate}
Finally we need to check the property \ref{prop:4prop_lambda_4} of Definition $\ref{def:abstract_fundamental_period}$.
As ${\pi^*}$ and ${\pi}$ have the same valuation (one), we have ${\pi^*}{\pi}^{-1}\in\mca{O}_F^{\times}$. In particular the equation 
$$\frac{\Fro^f(c)}{c} = \frac{\pi^*}{\pi}$$
has a solution $c\in (\widehat{F}\nr)^{\times}$, see for example \cite[Chap.~V, \S2, Lemma~(2.1)]{MR1697859}.   
Note that if $F/F_0$ is unramified and $\pi$ is in $F_0$, we can take $c=1$.
We finish by computing 
$$ \Fro[\mathrm{cris}]^f\left(c\frac{\alpha}{\alpha^{*\cris}}\right) =
\Fro^f(c)\frac{\Fro[\mathrm{cris}]^f(\alpha)}{\Fro[\mathrm{cris}]^f(*\cris(\alpha))}=
c\frac{\pi^*}{\pi}\frac{\Fro[\mathrm{cris}]^f(\alpha)}{*\cris(\Fro[\mathrm{cris}]^f(\alpha))} =
c\frac{\pi^*}{\pi}\frac{\pi \alpha}{(\pi\alpha)^{*\cris}} = c\frac{\alpha}{\alpha^{*\cris}},  
$$	
where we used $\Fro[\mathrm{cris}]^f(\alpha)=\pi\alpha$ and Lemma \ref{lemma:Gamma_cris}\eqref{lemma:Gamma_cris:2}.
\end{dem}

\begin{remark} \label{rk:transcendente}
If $F/F_0$ is tamely ramified, we can choose a uniformizer $\pi\in F$, such that $\pi^*=-\pi$ and 
the constant $c$ can be chosen in a quadratic unramified extension of $F$, such that $\Fro^f(c) = -c$.
In general, the constant $c$ is algebraic over $F$ if and only if its orbit under $\Gal{\widehat{F}\nr}{F}$ 
is finite. Since $\Fro^f$ is a topological generator, the relation   \eqref{eq:Cocycle_wild} implies that $c$ is algebraic if and only if 
$\pi^*\pi^{-1}$ is a root of unity. 
When the break of the ramification filtration of $\Gal{F}{F_0}$ is large, it may happen that for any choice of a uniformizer $\pi$ in $F$, 
$\pi^*\pi^{-1}$ is never a root of unity ({\itshape e.g.} $F=\QQ_2(\zeta_8)$ and $F_0=\QQ_2(\zeta_4)$, where $\zeta_8$ is a primitive 8th-root of unity and $\zeta_8^2=\zeta_4$).
\end{remark}

\begin{cor}\label{cor:explicit_lambda}  
Let $\lambda$ be a fundamental period in $\BcrisF$,
then
$$\lambda \in F^{\times}\cdot c\frac{\alpha}{\alpha^*},$$ with $c$ and $\alpha$ as in Proposition $\ref{prop:explicit_lambda}$.  
\end{cor}
\begin{dem} 
This is an immediate consequence of Propositions~$\ref{prop:unicity}$ and $\ref{prop:explicit_lambda}$.  \end{dem}

\subsection{Proof of Theorem \ref{th:not_a_norm}}\label{proof:not_a_norm}
Let $\lambda\in\BcrisF^{\times}$ be the fundamental period defined in $\eqref{eq:the_fund_period}$.
By Proposition \ref{prop:4prop_lambda} it is a fundamental period  in the sense of Definition \ref{def:abstract_fundamental_period}.
By Proposition \ref{cor:lambda_lambda*} (or by Proposition \ref{burrata}), the period $ \lambda\lambda^{*}$ belongs to $F_0^{\times}$.
By Remark \ref{rem:not_a_notm}\eqref{rem:not_a_notm2} it is enough to prove that $ \lambda\lambda^*$ does not belong to $ N_{F/F_0}(F^\times)$, for $*=*\cris$, the endomorphism of \BcrisF\ defined in \ref{sb:BcrisF_more}.

By Corollary \ref{cor:explicit_lambda}, we have $\lambda=bc\tfrac{\alpha}{\alpha^*}$, for some $b\in F^{\times}$, a Lubin--Tate period $\alpha\in\BcrisF^{\times}$ and $c\in\widehat{F}\nr$ as in Proposition $\ref{prop:explicit_lambda}$.
Considering that $bb^*=N_{F/F_0}(b)\in F_0^{\times}$ is a norm we can suppose  $b=1$.
Following Proposition \ref{prop:explicit_lambda}, we need to treat separately  two cases, depending on whether the extension $F/F_0$ is unramified or ramified.
\begin{enumerate}
\item Assume that the extension $F/F_0$ is unramified. We have $*\cris^2=\Fro^f$, \emph{cf.} \ref{sb:BcrisF_more}. 
Choose $\pi$ in $F_0$ so that we can have $c=1$, \emph{cf.} \ref{prop:explicit_lambda}. We get 
$$\lambda\lambda^*
=\frac{\alpha}{(\alpha^*)}\frac{\alpha^*}{(\alpha^*)^*} 
= \frac{\alpha}{(\alpha^*)^*} =
\frac{\alpha}{\Fro^f(\alpha)}=\frac{\alpha}{\pi\alpha}=\frac{1}{\pi}.$$ 
Since $F/F_0$ is unramified every norm has even valuation in $F_0$, hence $\lambda\lambda^*$ cannot be a norm. 
\item 
Assume that the extension $F/F_0$ is (totally) ramified.  
In particular, the endomorphism $*\cris$ of \BcrisF\ is an involution, \emph{cf.} \ref{sb:BcrisF_more}.
We get
$$\lambda\lambda^*
= c\frac{\alpha}{(\alpha^*)}c^*\frac{\alpha^*}{(\alpha^*)^*} 
= cc^*= N_{\widehat{F}\nr/\widehat{F}_0\nr}(c),$$ 
where the last equality follows given that the restriction of $*\cris$ to $\widehat{F}\nr=\widehat{F}_a\nr\otimes_{\Fa}F\subset \BcrisF$ is the unique non trivial element in $\Gal{\widehat{F}\nr}{\widehat{F}_0\nr}$. 
Let $$(-,F/F_0)\colon F_0^\times \rightarrow \Gal{F}{F_0}$$ be the local reciprocity map. 
The equation $\Fro^f(c) c^{-1}= \pi^* \pi^{-1} $ from \eqref{eq:Cocycle_wild} fits exactly the hypothesis of
a theorem of Dwork, \emph{cf.}~\cite[Ch.~XIII, \S5, Cor. of Th.~2]{MR0354618}; 
whence we get
\begin{equation*}
\label{eq:Dwork_th}
\bigl(N_{\widehat{F}\nr/\widehat{F}_0\nr}(c),F/F_0\bigr) = *^{-1} = *.
\end{equation*} 
Since $*\not = \id[F]$, that exactly means that $N_{\widehat{F}\nr/\widehat{F}_0\nr}(c)\in F_0^{\times}$ is not a norm, \emph{i.e.} it does not belong to $N_{F/F_0}(F^\times)$, and that finishes the proof. {\hfill$\square$}
\end{enumerate}

\begin{remark} 
For $*=*\cris$, the proof above, as all the statements of this section,  does not use that $\lambda$ 
is coming from  a CM-quadratic space, \emph{cf.} Definition~\ref{def CM quad form}, but only the
properties of Definition \ref{def:abstract_fundamental_period}. (See also Remark \ref{rem:not_using_quad_forms}). 
\end{remark}

\section{Proof of the main theorem}\label{sect:proof_MT}

We now put all the ingredients together and prove Theorem \ref{mainthm}.
\subsection*{Proof of Theorem \ref{mainthm}} 
By Proposition \ref{padicreduction}, the result is reduced to a $p$-adic statement. In particular, following Proposition \ref{motive admissibility}, it is enough to study the $p$-adic étale realization and the Frobenius invariant part of the crystalline realization of the motive $M$. By this same proposition, this pair of quadratic spaces is an orthogonal supersingular pair, in the language of Definition \ref{def supersingular}. We can reformulate Proposition \ref{padicreduction} using    Definition \ref{definition good}: we have to show that our orthogonal superisingular pair is good (see also Remark \ref{qualcosa}).

 These two quadratic spaces are endowed with the action of $F \otimes_{\QQ} \Qp$, where $F$ is the number field given by hypothesis.
By Proposition \ref{prop simple object} we can use the action of $F \otimes_{\QQ} \Qp$ to decompose the two quadratic spaces and study them separately. More precisely it is enough to study the case where the two quadratic spaces are endowed with an action of a $p$-adic field (which is a factor of $F \otimes_{\QQ} \Qp$) and is stable by the involution  $*$ induced by $F$. By abuse of notation we will denote such a $p$-adic field again by $F$.

We have now an orthogonal supersingular pair endowed with an action of a $p$-adic field. Both quadratic space are now CM-quadratic spaces with respect to $F$ in the sense of Definition \ref{def CM quad form}. We can now apply Proposition \ref{comp decomp}  to our two quadratic spaces and to the ring $B=\BcrisF$. This constructs a period $\lambda \in B^{\times}$, well defined up to multiplication by an element of 
$F^{\times}$.

By Proposition \ref{burrata} applied to the endomorphism  $*=*\cris$ defined in \ref{sb:BcrisF_more},   the two orthogonal spaces are isomorphic if and only if $\lambda  \cdot  *(\lambda)$ lies in the group of norms $N_{F/F_0}(F^{\times}).$ 
By Corollary \ref{cime di rapa}, this norm condition is equivalent to the fact that the minimum of the Hodge polygon of the underlying filtered $\varphi$-module is an even number (see also Remark \ref{cozze}). Altogether we have  that the two orthogonal spaces are isomorphic if and only if the minimum of the Hodge polygon is even. By Lemma \ref{buona notte} this means precisely that the pair is good. {\hfill$\square$}


\begin{remark}
Note that Corollary \ref{cime di rapa} is a consequence of Theorem \ref{th:not_a_norm} whose proof is the object of Sections \ref{sect:p-adic_periods} and \ref{sect:LT_and_reciprocity}.
\end{remark}

\appendix
\section{Lubin--Tate filtered $\Fro$-modules}\label{sect:LubinTate}

Lubin--Tate periods were constructed by Colmez in \cite{MR1956055}, via a direct construction based on Lubin--Tate's formal group law \cite{MR172878}.
The goal of this appendix is to present to the reader an alternative construction of these periods as self-contained  as possible. 
Beyond \cite{MR1956055},  
there is a vast literature 
for the $p$-adic representations associated with Lubin--Tate groups and their periods, {\itshape cf.} for example \cites{MR1247996,fourquaux_these,MR2565906,MR4068300}. 

\smallskip
Let $\pi$ be a uniformizer of a $p$-adic field $F$.  The plan is the following : 
\begin{enumerate}
	\item Describe concretely a filtered $\Fro$-module $D=(D,\Fro[\pi],\Fil[\bullet](D_F))$ over $F$ using only semi-linear algebra data, see Definitions \ref{sb:def_D_pi} and \ref{sb:def_D_pi_Filtration}.
	\item Show that $D$ is weakly-admissible, see Proposition \ref{prop:D_is_admissible}; hence admissible by \cite{MR1779803}.
	\item Show that a Lubin--Tate period as in Definition \ref{def:Lubin-Tate_period} appears as period $\alpha=\alpha_{\pi}$  of $D$,
	see \eqref{eq:def_period_D_pi} and Proposition \ref{prop:check_Lubin-Tate_period}.
	\item Relate $\alpha$ to Colmez' original construction, hence $D$ to the Galois representation given by the Lubin--Tate character, see 
	Proposition \ref{prop:LT_char}.
\end{enumerate}

In this appendix we do not assume anymore that $F$ is endowed with a non trivial involution.

\smallskip
Keep the notation of Conventions \ref{sb:Conv_p-adic_fields} and \ref{sb:Conv_pHT}. In particular recall that we denoted by $\Fa$ the absolute unramified subfield of $F$, which has degree $f$ over $\Qp$.
Let $E_\pi(x)\in\Fa[x]$ be the minimal (monic) polynomial of $\pi$ over $\Fa$, which is an Eisenstein polynomial of degree $e$.


\begin{definition}[The $\Fro$-module]\label{sb:def_D_pi} 
Let us define a filtered $\Fro$-module: 
 we set 
$$ D\coloneqq {(\Fa)}^{ef},$$ as $\Fa$-vector space and we endow it 
 with the semi-linear Frobenius $\Fro[\pi]\colon D\rightarrow D$ given, in the canonical basis $\mathscr{C}\coloneqq(\mathfrak{e}_1,\ldots,\mathfrak{e}_{ef})$, by the 
bloc matrix 
\begin{equation}
\label{eq:sb:FROB_D_pi}
A\coloneqq \left(\begin{array}{c|c} 
0  & C \\
\hline
I_{e(f-1)} & 0 \end{array}\right) \in \mathrm{M}_{ef}(\Fa),
\end{equation}
where $I_{e(f-1)}\in M_{e(f-1)}(\Fa)$  is the identity matrix, $0$ stands for  the zero (rectangular) blocs, and 
     $C \in M_e(\Fa)$ is the companion matrix of the minimal polynomial $E_{\pi}(x)$ of $\pi$.  
\end{definition}

\begin{lemma}\label{lemma:F-action}
\begin{enumerate}
	\item \label{lemma:F-action_1} The $\Fro$-module $(D,\Fro[\pi])$ is endowed with an $F$-action, {\itshape i.e.} a linear action of $F$ on
	the $\Fa$-vector space $D$ which commutes with $\Fro[\pi]$.
	\item \label{lemma:F-action_2}
	The action of $\pi\in F$ is equal to the action of $\Fro[\pi]^f$ on $D$.
\end{enumerate}
\end{lemma}
\begin{proof}
Let us define an $F$-action on the $\Fro$-module $D$. For any square matrix $M\in M_e(\Fa)$, we denote by $\Fro[\Fa](M)$ the matrix on which the Frobenius $\Fro[\Fa]$ is applied to all entries.
We then define the diagonal bloc matrix
\begin{equation} \label{eq:F-action}
\widetilde{M}\coloneqq \mathrm{Diag}(M, \Fro[\Fa](M),\ldots ,\Fro[\Fa]^{f-1}(M) ) \in \mathrm{M}_{ef}(\Fa),
\end{equation}
where the diagonal blocs are the matrices $\Fro[\Fa]^i (M)$, for $i$ going from $0$ to $f-1$.
The $\Fa$-linear map 
$D\rightarrow D$ given by  $\widetilde{M}$ in the canonical basis $\mathscr{C}$ of $D$ commutes with the (semi-linear) Frobenius $\Fro[\pi]$ if and only if $$ \widetilde{M} A = A \Fro[\Fa](\widetilde{M}),$$
by Definition \ref{sb:def_D_pi}.
It is elementary to check that this holds exactly when $M$ commutes with the matrix $C$ from \eqref{eq:sb:FROB_D_pi}.
Since $C$ is cyclic, its commutator in $M_e(\Fa)$ is precisely $\Fa[C]\iso F$. This gives an $\Fa$-linear $F$-action on the $\Fro$-module
$D$ and proves the statement \eqref{lemma:F-action_1} above.

We can now prove \eqref{lemma:F-action_2}. First, note that by construction the action $\pi\in F$ is given by the matrix
$\widetilde{C}$  from \eqref{eq:F-action},
since $C$ is the companion matrix of the minimal polynomial of $\pi$. 

Now, since $\Fro[\pi]$ is semilinear, its $f$ power $\Fro[\pi]^f$ is $\Fa$-linear and it is given in the canonical basis by the 
(twisted) product 
\begin{equation}
A \cdot\Fro[\Fa](A) \cdots  \Fro[\Fa]^{f-1}(A).
\label{eq:twisted-prod}
\end{equation} 
 An elementary computation shows that the product \eqref{eq:twisted-prod} is equal to the diagonal bloc matrix 
\begin{equation}
\label{eq:sb:f_power_FROB_D_pi}
\mathrm{Diag}(C, \Fro[\Fa] (C),\ldots ,\Fro[\Fa]^{f-1} (C) ) \in \mathrm{M}_{ef}(\Fa),
\end{equation}
which is the matrix $\widetilde{C}$ from \eqref{eq:F-action}. As remarked above this shows exactly that $\Fro[\pi]^f$ acts as $\pi$.
%
%
\end{proof}

\begin{remark}\label{ciliegia} 
	By construction we have two field actions on $D$: the (left) $F$-action defined in Lemma $\mathrm{\ref{lemma:F-action}}\eqref{lemma:F-action_1}$,  which does commute with the Frobenius; and another one
	of $\Fa$ via its original structure of $\Fa$-vector space, which does not commute in general with the Frobenius.
	 To distinguish them we denote the latter as a right action as we did in Proposition 
	$\ref{prop decomp}$. Those actions clearly coincide on $\Qp\subset F$, hence $D$ is an   $F\otimes_{\Qp} \Fa$-module.
\end{remark}

\begin{lemma}\label{rem:libero}
\begin{enumerate}
\item\label{rem:libero_2} The $F\otimes_{\Qp} \Fa$-module $D$ is free of rank one.
\item\label{rem:libero_3}
Set $$D_F\coloneqq D\otimes_{\Fa}F.$$ It is a free $F\otimes_{\Qp}F$-module of rank one.
In particular there is a unique $F$-line 
\begin{equation*} 
W \subset D_F, 
\end{equation*}
where the right 
$F$-structure of $D_F$ coincide with the left one.
\end{enumerate}
\end{lemma}
\begin{proof}
The statement \eqref{rem:libero_3} follows directly from \eqref{rem:libero_2}. (In particular, the uniqueness of the line $W$ comes from the fact that $D_F$ has rank one). 

Let us show that $D$ is free of rank one. 
Indeed the ring $R\coloneqq F\otimes_{\Qp}\Fa$ is a product of fields, so  
the module $D$ splits as a product of vector spaces on those fields. 
We claim that those vector spaces have all the same dimension,  thus $D$ is free as $R$-module by \cite[AC.II, \S5.3, Proposition 5, p.113]{MR1727221}; furthermore, since $[F:\Qp]=\dim_{\Fa}(D)$, the rank of $D$ must be one.
Let us prove the claim. It follows from the fact that $D$ is endowed with a Frobenius: if we twist the action of $\Fa$ on $D$ by the absolute Frobenius of $\Fa$, then the Frobenius $\Fro[\pi]$ of $D$ induces  
 an $R$-linear isomorphism $\Phi_{\pi}\colon \Fro^*D\coloneqq D \otimes_{\Fro} \Fa\rightarrow D$, mapping $m\otimes a$ to $\Fro[\pi](m)\cdot a$.
Since $F\supset \Fa$ and $\Fa/\Qp$ is a Galois extension with Galois group generated by $\Fro$, we see that the ring $R$ splits exactly as a product of $f$  fields (all abstractly isomorphic to $F$); thus the Frobenius $\id[F]\otimes\Fro$ acts transitively on $\Spec(R)$ as an $f$-cycle.
Therefore the claim follows from the existence of the isomorphism  $\Phi_{\pi}$. 
%
%
\end{proof}

\begin{definition}[The filtered $\Fro$-module]\label{sb:def_D_pi_Filtration} 
Let us define a filtration on $D_F = D\otimes_{\Fa}F.$
We set 

$$ \Fil (D_F) \coloneqq\begin{cases} D_F & \text{ if }i\leq 0; \\
                                                W   & \text{ if }i=1; \\
                                              0 & \text{ if }i\geq 2;\end{cases}$$
where the $F$-line $W$ is constructed in Lemma \ref{rem:libero}\eqref{rem:libero_3}. 
The datum $$D_{\pi}\coloneqq (D,\Fro[\pi], \Fil[\bullet] (D_F))$$ altogether with the $F$-action, defined in  
\ref{lemma:F-action}--\ref{rem:libero},
form a filtered $\Fro$-module $D_{\pi}$ over $F$ with $F$-coefficients, {\itshape cf.} \cite{MR1944572}\footnote{In their definition, set  $N\coloneqq 0$  as monodromy operator.}, that 
  we call  \emph{the Lubin--Tate filtered $\Fro$-module associated with} 
$\pi\in F$. We may denote it $D$ for brevity.
\end{definition}

\begin{prop} \label{prop:D_is_admissible} 
 The Lubin--Tate filtered $\Fro$-module $D$ is admissible.  Its Newton and Hodge polygon are given as in the picture below.
$$     
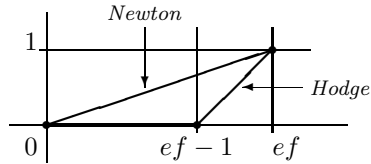
\setlength{\unitlength}{1mm}
\begin{picture}(45,23)
\linethickness{0.15mm}
\put(0,5){\line(1,0){45}}
\put(5,0){\line(0,1){20}}
\put(4,15){\line(1,0){36}}
\put(25,4){\line(0,1){14}}
\put(35,4){\line(0,1){14}}
\thicklines
\put(5,5){\line(1,0){20}}
\put(5,5){\line(3,1){30}}
\put(25,5){\line(1,1){10}}
\put(5,5){\circle*{1}}
\put(25,5){\circle*{1}}
\put(35,15){\circle*{1}}
\put(2,1){$0$}
\put(2,15){$1$}
\put(20,1){$ef-1$}
\put(35,1){$ef$}
\thinlines
\linethickness{0.075mm}
\put(18,18){\vector(0,-1){8}}
\put(13,19){\text{\footnotesize{Newton}}}
\put(39,10){\vector(-1,0){8}}
\put(40,9){\text{\footnotesize{Hodge}}}
\end{picture}
$$
\end{prop}
\begin{dem} By the definition of the filtration, the Hodge--Tate slopes are: the slope zero with multiplicity $ef-1$ and the slope one with multiplicity one. 
Let us compute the Newton slopes. The slope of the Frobenius $\Fro[\pi]$ are equal to the slopes of $\Fro[\pi]^f$ divided by $f$,
{\itshape cf.} \cite[2.1.3]{MR0498577}.
Since $\Fro[\pi]^f$ is an $\Fa$-linear endomorphism of $D$, its slopes are given by the $p$-adic valuation of the eigenvalues of $\Fro[\pi]^f$, where the $p$-adic valuation $v$ is normalized by $v(\Fa^{\times})=\ZZ$, {\itshape cf. loc.cit.}.
The eigenvalues of the matrix \eqref{eq:sb:f_power_FROB_D_pi} are the eigenvalues of the diagonal bloc matrices  $\Fro[\Fa]^i(C)$.
The matrix $\Fro[\Fa]^i(C)$ is the companion matrix of the polynomial $\Fro[\Fa]^i(E_{\pi}(x))\in\Fa[x]$ and its eigenvalues are the roots 
of it. For any $i$, this is an Eisenstein polynomial of degree $e$; whence,
by using the theory of the Newton polygon (for polynomials with coefficients on local fields, {\itshape cf.} \cite[Ch.~I, \S6, Theorem 6.1]{MR1274045}), we get that all the roots of  $\Fro[\Fa]^i(E_{\pi}(x))$ have the same valuation $1/e$.   
Therefore,  the Frobenius $\Fro[\pi]$ has only one slope, equal to $1/ef$, with multiplicity $ef$, and the $\Fro$-module $D$ must be absolutely irreducible: indeed
its Newton polygon does not meet points with integral coordinates, except for the vertexes $(0,0)$ and $(ef,1)$, see the picture above.
Finally, the Newton polygon of $D$ has the same height as the Hodge polygon and lies above it; by irreducibility there are no nontrivial submodules; thus $D$ satisfies the condition of weakly admissibility. Therefore it is admissible by \cite{MR1779803}.   
\end{dem}

\begin{remark}
The $\Fro$-module $D$ is actually an $\Fa$-form of the irreducible 
$\Fro$-module of slope $1/ef$ over $\widehat{F}\nr$, in the Dieudonne--Manin classification, {\itshape cf.} for example \cite[Theorem 2]{MR0498577}. Here 
$\widehat{F}\nr$ denotes the completion of maximal unramified extension of $F$.
Its endomorphisms as $\Fro$-module form the central division algebra over $\Qp$ with invariant $1/{ef}$, {\itshape cf.}
\cite[Chapitre VI, \S3, Lemma 3.3.2.2]{MR0338002}; 
whereas its endomorphisms as \emph{filtered} $\Fro$-module gives the maximal abelian subfield $F$ of this algebra.
\end{remark}

\subsection{} Before attaching a period to $D$, let us  recall  the contravariant Fontaine formalism, 
see \cite[5.3.7]{MR1293972}.
Set $$ \Vcrisv(D)\coloneqq \Hom{\text{fil-}\Fro\text{-mod}}{D}{\Bcris}, $$
where the morphisms are taken in the category of filtered $\Fro$-modules over $F$, {\itshape cf. loc.~cit.}~4.3.3.
The \Qp-vector space $\Vcrisv(D)$  is a crystalline representation of $G_F$  which is 
 of dimension $ef$ over $\Qp$, since  $D$ is admissible. Moreover, it is endowed with an action of $F$, thus it is an $F$-vector space necessarily of dimension one.

 
\subsection{Definition of the period} \label{def_period_D_pi}
Keep previous notation; 
moreover,
we use the notation of \ref{sb:def_nu} and \ref{sb:BcrisF_more}, excepted that we are not assuming anymore that $F$ is endowed with a non trivial involution. 

Let us  attach a period $\alpha_{\pi}$ to $D=D_{\pi}$.
Any morphism $x\in \Vcrisv(D)$ can be seen as the element $$x_F\coloneqq x\otimes \iota \in D^{\vee}\dR\coloneqq\Hom{F}{D_{F}}{\Bdr},$$ where $\iota\colon F\hookrightarrow \Bdr$ is the embedding fixed in \ref{sb:def_nu}.
It is clear by construction that for every $x\in\Vcrisv(D)$, we have $x_F(D_F)\subset \BcrisF\cap\Bdr^+$.
Let us choose a basis  $w$ of the line $W=\Fil[1](D_F)$ over $F$ and a basis  
$x\in\Vcrisv(D)$ over $F$. We define a period 
\begin{equation}
\label{eq:def_period_D_pi}
\alpha_{\pi}\coloneqq x_F(w) \in \BcrisF\cap\Fil[1]\Bdr. 
\end{equation}
It does not depend on the choices of $w$ and $x$ up to a constant in $F^{\times}$ and we call it the period associated with 
the filtered $\Fro$-module $D_{\pi}$. We will see in Proposition \ref{prop:check_Lubin-Tate_period} that it is a Lubin--Tate period.

\subsection{}\label{sb:def_w_sigma}
Recall from \ref{sb:def_nu} that we have denoted by $\tilde{F}\subset\Bdr$ the Galois closure of $F$ and by $\Gamma$ the set of $\Qp$-embeddings of $F$ in $\tilde{F}$.
For any $\sigma\in \Gamma$, consider the map 
$$\Fro[\pi]^{\nu(\sigma)}\otimes \sigma \colon D\otimes_{\Fa} F \rightarrow D \otimes_{\Fa} \tilde{F},$$
and set $w_{\sigma}\coloneqq (\Fro[\pi]^{\nu(\sigma)}\otimes \sigma) (w) \in 
D_{\tilde{F}}\coloneqq D_F\otimes_F \tilde{F} $, where $w$ is the basis of $\Fil[1](D_F)$ chosen in \eqref{eq:def_period_D_pi}.  
 We have $w_{\sigma}\not = 0$ since $w\not = 0$ and $\Fro[\pi]^{\nu(\sigma)}\otimes \sigma$ is injective. 
By definition of $\iota$ we get $w_{\iota}=w\otimes 1$.

\begin{lemma}\label{lemma:normal_basis} For any $\cf\in F$ and $d\in D_{\tilde{F}}=D\otimes_{\Fa}\tilde{F}$ denote by $\cf(d)$ the action induced by the action of $F$ on $D$ and by $d\cdot \cf$ the action by multiplication by the scalar $\cf\in\tilde{F}$ on the right. 
Let $w_{\sigma}\in D_{\tilde{F}}$ be the element defined $\mathrm{\ref{sb:def_w_sigma}}$. The following are true.
\begin{enumerate}
\item \label{lemma:normal_basis_1}
For every $\sigma\in \Gamma$, we have $\Fro[\pi]^f (w_{\sigma})=\sigma(\pi) w_{\sigma}$.
	\item \label{lemma:normal_basis_1.5} For every $\sigma\in \Gamma$ and for every $\cf\in F$, we have 
	$ \cf(w_{\sigma})=w_{\sigma}\cdot \sigma(\cf)$.
	\item \label{lemma:normal_basis_2} 
	The family $(w_{\sigma})_{\sigma\in\Gamma}$  forms a basis of $D_{\tilde{F}}$ as vector space over $\tilde{F}$. 
\end{enumerate}
\end{lemma}
\begin{dem} 
The proof of $\eqref{lemma:normal_basis_1}$ is a formal argument:
 $$\Fro[\pi]^f (w_{\sigma})= (\Fro[\pi]^{f+\nu(\sigma)}\otimes \sigma) (w)=
(\Fro[\pi]^{\nu(\sigma)}\otimes \sigma) (\Fro[\pi]^f(w))= (\Fro[\pi]^{\nu(\sigma)}\otimes \sigma) (\pi w)
 =\sigma(\pi) w_{\sigma}.$$
Let us prove \eqref{lemma:normal_basis_1.5}: 
$$ \cf(w_{\sigma})=\cf((\Fro[\pi]^{\nu(\sigma)}\otimes \sigma) (w))=(\cf\Fro[\pi]^{\nu(\sigma)}\otimes \sigma) (w)=
(\Fro[\pi]^{\nu(\sigma)}\cf\otimes \sigma) (w) = (\Fro[\pi]^{\nu(\sigma)}\otimes \sigma) (\cf(w)),$$
since the $F$-action on $D$ commutes with $\Fro[\pi]$ by construction. Now by definition of $w$, we have $\cf(w)=w\cdot \cf$, hence 
we finish by 
$$ (\Fro[\pi]^{\nu(\sigma)}\otimes \sigma) (\cf(w)) = (\Fro[\pi]^{\nu(\sigma)}\otimes \sigma) (w\cdot \cf) = w_{\sigma}\cdot \sigma(\cf).$$
Point \eqref{lemma:normal_basis_2}  follows from \eqref{lemma:normal_basis_1.5} : the $w_{\sigma}$ are eigenvectors of multiplicity one diagonalizing the left action of $F$ on $D_{\tilde{F}}$.  
\end{dem}

\begin{prop} \label{prop:check_Lubin-Tate_period}
 Let $\pi\in F$ be a uniformizer and $D$ the associated Lubin--Tate  filtered $\Fro$-module from Definition  
$\mathrm{\ref{sb:def_D_pi_Filtration}}$. Then the element $\alpha\coloneqq\alpha_{\pi}\in\BcrisF\cap\Fil[1]\Bdr$ associated with it in \eqref{eq:def_period_D_pi}  is a Lubin--Tate period relative to $\pi\in F$ in the sense of Definition $\ref{def:Lubin-Tate_period}$.
\end{prop}
\begin{dem} %
We compute 
$$\Fro[\mathrm{cris}]^f(\alpha) = ((\Fro[\mathrm{cris}]^f\otimes \id[F])\circ (x\otimes \iota))(w) =((x\otimes \iota)\circ (\Fro[\pi]^f\otimes \id[F]))(w)= (x\otimes \iota)(w \cdot \pi)= \pi \alpha,$$
so  the condition \eqref{def:Lubin-Tate_period:3} of Definition \ref{def:Lubin-Tate_period} holds.

Let us check the other properties of Definition \ref{def:Lubin-Tate_period}.
By construction $\vdR(\alpha)\geq 1$ and,  
for all $\sigma\in \Gamma\backslash\{ \iota\}$,  we compute
\begin{equation}\label{eq:ricotta}
\sigma\cris(\alpha)= ((\Fro[\mathrm{cris}]^{\nu(\sigma)}\otimes \sigma)\circ (x\otimes\iota))(w)=
((x\otimes\id[\tilde{F}]) \circ (\Fro[\pi]^{\nu(\sigma)}\otimes \sigma))(w)=	(x \otimes \id[\tilde{F}]) (w_{\sigma}),
\end{equation}
and in  particular  $\vdR(\sigma\cris(\alpha))\geq 0$, since 
$(x\otimes \id[\tilde{F}])(D_{\tilde{F}})\subset \BcrisFtilde\cap\Bdr^+$.

Let us define
\begin{equation} \label{eq:def_P}
	P\coloneqq\prod_{\sigma\in\Gamma} \sigma\cris(\alpha) 
	\in \BcrisFtilde.
\end{equation}
We claim the following:
\begin{enumerate}
	\item $P$ belongs to $\BcrisF$.
	\item $P$ is invertible in $\BcrisF$.
	\item  $\vdR(P) =1$.
\end{enumerate} 
These claims conclude the proof.  
Indeed the equality $\sum_{\sigma\in\Gamma}\vdR(\sigma\cris(\alpha))=\vdR(P) =1$ implies 
that the inequalities we have shown above on the $\vdR(\sigma\cris(\alpha))$  must be equalities.
Moreover, $\alpha$ is invertible in $\BcrisF$, since it divides $P$ which is invertible. 

Let us first prove claim $(1)$.
Consider the Galois group $G\coloneqq\Gal{\tilde{F}}{F'}$, where $F'\coloneqq\tilde{F}_a F$. 
Let $G$ act on $\BcrisFtilde=\Bcris\otimes_{\tilde{F}_a}\tilde{F}$ as the identity on \Bcris\ and via its natural action on the right.
By its very construction $P$ is stable under $G$ hence it belongs to $$(\BcrisFtilde)^G= B_{\mathrm{cris},F'}=\Bcris\otimes_{\tilde{F}_a}\tilde{F}_aF=
\Bcris\otimes_{\Fa}F  =\BcrisF.$$

To show claims $(2)$ and $(3)$ consider the 
 determinant $\Delta$ of the filtered $\Fro$-module $D$ 
$$ \Delta \coloneqq \det_{\Fa}(D) = \bigwedge^{ef} D. $$
It is  admissible of slope one (both Newton and Hodge), hence 
$\Vcrisv(\Delta)$ is a rank one crystalline representation of Hodge-Tate weight one. 
To conclude, it is enough to show that, up to a constant in $F^{\times}$, $P$ is the period $\delta\in\Bcris^{\times}$ of $ \Delta$. 

The left $F$-action on $V_{\tilde{F}}\coloneqq\Vcrisv(D)\otimes_{\Qp}\tilde{F}$ decomposes the element $x\otimes \id[\tilde{F}] $  as sum of eigenvectors $x\otimes \id[\tilde{F}]=\sum_{\sigma\in\Gamma} x_{\sigma}$.
By Lemma \ref{lemma:normal_basis}\eqref{lemma:normal_basis_2}, $(w_{\sigma})_{\sigma\in\Gamma}$  is a basis of $D_{\tilde{F}}$, and similarly 
$(x_{\sigma})_{\sigma\in\Gamma}$ is a basis of $V_{\tilde{F}}$. 
We have $x_{\tau}(w_{\sigma})=0$ whenever $\tau\not = \sigma$, because $x_{\tau}$ and $w_{\sigma}$ live in eigenlines of different eigenvalues.
Hence, up to a constant $u$ in $\tilde{F}^{\times}$, we can compute the period $\delta$ 
by using these bases,  as  
$$ u\delta = \prod_{\sigma\in\Gamma} (x_{\sigma}\otimes \id[\tilde{F}])(w_{\sigma}) = \prod_{\sigma\in\Gamma} (x\otimes \id[\tilde{F}])(w_{\sigma}), 
$$
which is equal to $P$ by  \eqref{eq:ricotta}.
Finally, by claim $(1)$ the constant $u$ must belong to $F^{\times}$.
\end{dem}

\begin{remark}  \label{Remark_prod_alpha}
	The period $\delta$ of $\Delta$ can be computed directly by taking the determinant of the matrix $A$ of \eqref{eq:sb:FROB_D_pi}. This leads to a finer expression of $P$ as 
	\begin{equation}
			P 
			 =   u c_{\eta} t \in \BcrisF^{\times},
			\label{eq:alpha_divide_t}
		\end{equation}
	where $u\in {F}^{\times}$, 
	$t$ is the period of $\Qp(1)$, {\itshape cf.} Convention~\ref{sb:Conv_pHT}\eqref{Bcris_3}, 
	and $c_{\eta}\in W(\overline{\mathbb{F}}_p)^{\times} $  
	is the period of an unramified character $\eta\colon G_F\rightarrow \Zp^{\times}$. Notice that such a period $c_{\eta}$ lies in $W(\overline{\mathbb{F}}_p)^{\times} \subset \Bcris^{\times}$  since the cohomology group $H^1(G_{k_F}, W(\overline{\mathbb{F}}_p)^{\times})$ is trivial by \cite[III-33 Lemma]{MR1484415}. 
	
	More precisely, the period $c_{\eta}$ can be chosen in $W(\overline{\mathbb{F}}_p)^{\times}$, up to an invertible in $\Zp$, as the solution of 
	the equation 
	\begin{equation}
	\frac{\Fro (c_{\eta})}{c_{\eta}}= \frac{\det A}{p},  
\end{equation}
	which exists by \cite[Chap.~V, \S2, Lemma~(2.1)]{MR1697859} since ${\det A}$ 
	is a uniformizer of $W(k_F)$. 
	
	
	Furthermore, notice that the action of the Frobenius $\Fro^{f}$ of $\BcrisF$ on $P$ is easy to compute as 
	\begin{equation}
	\Fro^f(P) = N_{F/\Qp}(\pi) P.	
	\end{equation}
\end{remark}

%

\begin{prop}\label{prop:LT_char}
Let $\pi\in F$ be a uniformizer and $D$ the associated Lubin--Tate filtered $\Fro$-module, from Definition 
\emph{\ref{sb:def_D_pi_Filtration}}. Then the crystalline representation $\Vcrisv(D)$ is given by the Lubin--Tate character associated with $\pi\in F$. 
\end{prop}
\begin{dem} Let $\mathfrak{F}$ be a Lubin--Tate group over $\mca{O}_F$ associated with $\pi$. Let $T$ be its Tate module, 
given by the Lubin--Tate character $\psi\colon G_F\rightarrow \mca{O}^{\times}_F$; set $V\coloneqq T[1/p]$. We want to show that 
$\Vcrisv(D)=V$. Colmez constructs an element $t_{\mathfrak{F}}\in\BcrisF$, satisfying Definition \ref{def:Lubin-Tate_period} and 
such that for every $g\in G$, $g(t_{\mathfrak{F}})=\psi(g)t_{\mathfrak{F}}$, see for example \cite[\S3.6]{fourquaux_these}. 
Let $\alpha=\alpha_{\pi}\in\BcrisF^{\times}$ the Lubin--Tate period in Proposition \ref{prop:check_Lubin-Tate_period}.
By Lemma \ref{lemma:discesa_F} the period $t_{\mathfrak{F}}\alpha^{-1}$ belongs to $F^{\times}$, whence we have also, for every $g\in G_F$, $g(\alpha)=\psi(g)\alpha$. 
The map $\Vcrisv(D)\rightarrow \BcrisF$, defined by $x\mapsto (x\otimes \iota)(w)$, is $G_F$-equivariant by definition of the action of 
$G_F$ on $\Vcrisv(D)$ and it is $F$-linear by definition of $w$, {\itshape cf.}~\eqref{eq:def_period_D_pi}.
Therefore this map identifies $\Vcrisv(D)$ to $F\cdot \alpha \subset \BcrisF$ and that completes the proof.
\end{dem}
\begin{remark}
The representation $\Vcrisv(D)$ is an example of $F$-crystalline representation, {\itshape cf.} \cite{MR2565906}*{Introduction}. 
For a more general construction relating $F$-crystalline representations to their filtered $\Fro$-modules see 
\citelist{\cite{MR2565906}*{\S(3.3)} \cite{MR3505135}*{\S4.1}}. 
\end{remark}



\begin{bibdiv}
\begin{biblist}


\bib{Agug}{article}{
eprint = {https://arxiv.org/abs/2401.17445},
	author={Agugliaro, Thomas},
	title={Examples for the standard conjecture of Hodge type},
journal = {arXiv},
	year={2024},
}

\bib{MR4199442}{article}{
   author={Ancona, Giuseppe},
   title={Standard conjectures for abelian fourfolds},
   journal={Invent. Math.},
   volume={223},
   date={2021},
   number={1},
   pages={149--212},
   issn={0020-9910},
   review={\MR{4199442}},
   doi={10.1007/s00222-020-00990-7},
}

\bib{AncFra}{article}{
  
  eprint = {https://arxiv.org/abs/2207.09213},
  
  author = {Ancona, Giuseppe},
	author = {Fratila, Dragos},
  
  
  title = {Algebraic classes in mixed characteristic and André's p-adic periods},
  
  journal = {arXiv},
  
  year = {2022},
  
}


\bib{MR2115000}{book}{
   author={Andr\'{e}, Yves},
   title={Une introduction aux motifs (motifs purs, motifs mixtes,
   p\'{e}riodes)},
   language={French, with English and French summaries},
   series={Panoramas et Synth\`eses [Panoramas and Syntheses]},
   volume={17},
   publisher={Soci\'{e}t\'{e} Math\'{e}matique de France, Paris},
   date={2004},
   pages={xii+261},
   isbn={2-85629-164-3},
   review={\MR{2115000}},
}

\bib{And05}{article}{
   author={Andr\'{e}, Yves},
   title={Motifs de dimension finie (d'apr\`es S.-I. Kimura, P.
   O'Sullivan$\dots$)},
   language={French, with French summary},
   note={S\'{e}minaire Bourbaki. Vol. 2003/2004},
   journal={Ast\'{e}risque},
   number={299},
   date={2005},
   pages={Exp. No. 929, viii, 115--145},
   issn={0303-1179},
   review={\MR{2167204}},
}


\bib{MR4068300}{article}{
   author={Berger, Laurent},
   author={Schneider, Peter},
   author={Xie, Bingyong},
   title={Rigid character groups, Lubin-Tate theory, and
   $(\varphi,\Gamma)$-modules},
   journal={Mem. Amer. Math. Soc.},
   volume={263},
   date={2020},
   number={1275},
   pages={v+79},
   issn={0065-9266},
   isbn={978-1-4704-4073-2},
   isbn={978-1-4704-5658-0},
   review={\MR{4068300}},
   doi={10.1090/memo/1275},
}


\bib{MR0274237}{book}{
   author={Bourbaki, Nicolas},
   title={\'{E}l\'{e}ments de math\'{e}matique. Alg\`ebre. Chapitres 1 \`a 3},
   language={French},
   publisher={Hermann, Paris},
   date={1970},
   pages={xiii+635 pp. (not consecutively paged)},
   review={\MR{0274237}},
}

\bib{MR1727221}{book}{
   author={Bourbaki, Nicolas},
   title={Commutative algebra. Chapters 1--7},
   series={Elements of Mathematics (Berlin)},
   note={Translated from the French;
   Reprint of the 1989 English translation},
   publisher={Springer-Verlag, Berlin},
   date={1998},
   pages={xxiv+625},
   isbn={3-540-64239-0},
   review={\MR{1727221}},
}


\bib{MR1944572}{article}{
   author={Breuil, Christophe},
   author={M\'{e}zard, Ariane},
   title={Multiplicit\'{e}s modulaires et repr\'{e}sentations de ${\rm GL}_2({\bf
   Z}_p)$ et de ${\rm Gal}(\overline{\bf Q}_p/{\bf Q}_p)$ en $l=p$},
   language={French, with English and French summaries},
   note={With an appendix by Guy Henniart},
   journal={Duke Math. J.},
   volume={115},
   date={2002},
   number={2},
   pages={205--310},
   issn={0012-7094},
   review={\MR{1944572}},
   doi={10.1215/S0012-7094-02-11522-1},
}

\bib{MR4015230}{article}{
	author={Buskin, Nikolay},
	title={Every rational Hodge isometry between two $K3$ surfaces is
		algebraic},
	journal={J. Reine Angew. Math.},
	volume={755},
	date={2019},
	pages={127--150},
	issn={0075-4102},
	review={\MR{4015230}},
	doi={10.1515/crelle-2017-0027},
}

\bib{MR3505135}{article}{
   author={Cais, Bryden},
   author={Liu, Tong},
   title={On $F$-crystalline representations},
   journal={Doc. Math.},
   volume={21},
   date={2016},
   pages={223--270},
   issn={1431-0635},
   review={\MR{3505135}},
   doi={10.1007/s12204-016-1722-3},
}


\bib{MR1247996}{article}{
   author={Colmez, Pierre},
   title={P\'{e}riodes des vari\'{e}t\'{e}s ab\'{e}liennes \`a multiplication complexe},
   language={French},
   journal={Ann. of Math. (2)},
   volume={138},
   date={1993},
   number={3},
   pages={625--683},
   issn={0003-486X},
   review={\MR{1247996}},
   doi={10.2307/2946559},
}


\bib{MR1956055}{article}{
   author={Colmez, Pierre},
   title={Espaces de Banach de dimension finie},
   language={French, with English and French summaries},
   journal={J. Inst. Math. Jussieu},
   volume={1},
   date={2002},
   number={3},
   pages={331--439},
   issn={1474-7480},
   review={\MR{1956055}},
   doi={10.1017/S1474748002000099},
}


\bib{MR1779803}{article}{
   author={Colmez, Pierre},
   author={Fontaine, Jean-Marc},
   title={Construction des repr\'{e}sentations $p$-adiques semi-stables},
   language={French},
   journal={Invent. Math.},
   volume={140},
   date={2000},
   number={1},
   pages={1--43},
   issn={0020-9910},
   review={\MR{1779803}},
   doi={10.1007/s002220000042},
}

%


\bib{MR98079}{article}{
   author={Dwork, Bernard},
   title={Norm residue symbol in local number fields},
   journal={Abh. Math. Sem. Univ. Hamburg},
   volume={22},
   date={1958},
   pages={180--190},
   issn={0025-5858},
   review={\MR{98079}},
   doi={10.1007/BF02941951},
}

\bib{MR1274045}{book}{
   author={Dwork, Bernard},
   author={Gerotto, Giovanni},
   author={Sullivan, Francis J.},
   title={An introduction to $G$-functions},
   series={Annals of Mathematics Studies},
   volume={133},
   publisher={Princeton University Press, Princeton, NJ},
   date={1994},
   pages={xxii+323},
   isbn={0-691-03681-0},
   review={\MR{1274045}},
}



\bib{MR1463696}{article}{
   author={Faltings, Gerd},
   title={Crystalline cohomology and $p$-adic Galois-representations},
   conference={
      title={Algebraic analysis, geometry, and number theory},
      address={Baltimore, MD},
      date={1988},
   },
   book={
      publisher={Johns Hopkins Univ. Press, Baltimore, MD},
   },
   date={1989},
   pages={25--80},
   review={\MR{1463696}},
}


\bib{MR657238}{article}{
   author={Fontaine, Jean-Marc},
   title={Sur certains types de repr\'{e}sentations $p$-adiques du groupe de
   Galois d'un corps local; construction d'un anneau de Barsotti-Tate},
   language={French},
   journal={Ann. of Math. (2)},
   volume={115},
   date={1982},
   number={3},
   pages={529--577},
   issn={0003-486X},
   review={\MR{657238}},
   doi={10.2307/2007012},
}
	

\bib{MR1293971}{article}{
   author={Fontaine, Jean-Marc},
   title={Le corps des p\'eriodes $p$-adiques},
   note={With an appendix by Pierre Colmez;
   P\'eriodes $p$-adiques (Bures-sur-Yvette, 1988)},
   journal={Ast\'erisque},
   number={223},
   date={1994},
   pages={59--111},
   issn={0303-1179},
   review={\MR{1293971}},
}


\bib{MR1293972}{article}{
   author={Fontaine, Jean-Marc},
   title={Repr\'{e}sentations $p$-adiques semi-stables},
   language={French},
   note={With an appendix by Pierre Colmez;
   P\'{e}riodes $p$-adiques (Bures-sur-Yvette, 1988)},
   journal={Ast\'{e}risque},
   number={223},
   date={1994},
   pages={113--184},
   issn={0303-1179},
   review={\MR{1293972}},
}


\bib{fontaine_ouyang}{webpage}{
  title={Theory of $p$-adic Galois representations},
  author={Fontaine, Jean-Marc},
	author={Ouyang, Yi},
	date={2022-5-24},
	URL = {http://staff.ustc.edu.cn/~yiouyang/research.html},
}

\bib{fourquaux_these}{thesis}{
  TITLE = {{Logarithme de Perrin-Riou pour des extensions associ{\'e}es {\`a} un groupe de Lubin-Tate}},
  AUTHOR = {Fourquaux, Lionel},
  URL = {https://tel.archives-ouvertes.fr/tel-00011919},
  SCHOOL = {{Universit{\'e} Pierre et Marie Curie - Paris VI}},
  YEAR = {2005-12-12},
  TYPE = {Theses},
  eprint = {https://tel.archives-ouvertes.fr/tel-00011919/file/these.pdf},
}

\bib{Gro}{article}{
   author={Grothendieck, A.},
   title={Standard conjectures on algebraic cycles},
   conference={
      title={Algebraic Geometry},
      address={Internat. Colloq., Tata Inst. Fund. Res., Bombay},
      date={1968},
   },
   book={
      series={Tata Inst. Fundam. Res. Stud. Math.},
      volume={4},
      publisher={Published for the Tata Institute of Fundamental Research,
   Bombay by Oxford University Press, London},
   },
   date={1969},
   pages={193--199},
   review={\MR{0268189}},
}

\bib{MR1070716}{book}{
   author={Ireland, Kenneth},
   author={Rosen, Michael},
   title={A classical introduction to modern number theory},
   series={Graduate Texts in Mathematics},
   volume={84},
   edition={2},
   publisher={Springer-Verlag, New York},
   date={1990},
   pages={xiv+389},
   isbn={0-387-97329-X},
   review={\MR{1070716}},
   doi={10.1007/978-1-4757-2103-4},
}



\bib{MR4241794}{article}{
   author={Ito, Kazuhiro},
   author={Ito, Tetsushi},
   author={Koshikawa, Teruhisa},
   title={CM liftings of $K3$ surfaces over finite fields and their
   applications to the Tate conjecture},
   journal={Forum Math. Sigma},
   volume={9},
   date={2021},
   pages={Paper No. e29, 70},
   review={\MR{4241794}},
   doi={10.1017/fms.2021.24},
}


\bib{IIK}{article}{
   author={Ito, Kazuhiro},
	 author={Ito, Tetsushi}, 
	 author={Koshikawa, Teruhisa},
   title = {The Hodge standard conjecture for self-products of K3 surfaces},
   eprint={https://arxiv.org/abs/2206.10086},
   doi={10.48550/ARXIV.2206.10086},
  publisher = {arXiv},
  year = {2022},
}

\bib{MR2125735}{article}{
    AUTHOR = {Ito, Tetsushi},
     TITLE = {Weight-monodromy conjecture for {$p$}-adically uniformized
              varieties},
   JOURNAL = {Invent. Math.},
    VOLUME = {159},
      YEAR = {2005},
    NUMBER = {3},
     PAGES = {607--656},
       DOI = {10.1007/s00222-004-0395-y},
}



\bib{MR1150598}{article}{
   author={Jannsen, Uwe},
   title={Motives, numerical equivalence, and semi-simplicity},
   journal={Invent. Math.},
   volume={107},
   date={1992},
   number={3},
   pages={447--452},
   issn={0020-9910},
   review={\MR{1150598}},
   doi={10.1007/BF01231898},
}

\bib{Kahn}{article}{
	author={Kahn, Bruno},
	author={Murre, Jacob P.},
	author={Pedrini, Claudio},
	title={On the transcendental part of the motive of a surface},
	conference={
		title={Algebraic cycles and motives. Vol. 2},
	},
	book={
		series={London Math. Soc. Lecture Note Ser.},
		volume={344},
		publisher={Cambridge Univ. Press, Cambridge},
	},
	isbn={978-0-521-70175-4},
	isbn={0-521-70175-9},
	date={2007},
	pages={143--202},
	review={\MR{2187153}},
}

\bib{MR563463}{article}{
   author={Katz, Nicholas M.},
   title={Slope filtration of $F$-crystals},
   conference={
      title={Journ\'{e}es de G\'{e}om\'{e}trie Alg\'{e}brique de Rennes},
      address={Rennes},
      date={1978},
   },
   book={
      series={Ast\'{e}risque},
      volume={63},
      publisher={Soc. Math. France, Paris},
   },
   date={1979},
   pages={113--163},
   review={\MR{563463}},
}

\bib{MR0498577}{article}{
   author={Katz, Nicholas M.},
   title={Travaux de Dwork},
   language={French, with English summary},
   conference={
      title={S\'{e}minaire Bourbaki, 24\`eme ann\'{e}e (1971/1972), Exp. No. 409},
   },
   book={
      series={Lecture Notes in Math., Vol. 317},
      publisher={Springer, Berlin},
   },
   date={1973},
   pages={167--200},
   review={\MR{0498577}},
}





%

\bib{MR2565906}{article}{
   author={Kisin, Mark},
   author={Ren, Wei},
   title={Galois representations and Lubin-Tate groups},
   journal={Doc. Math.},
   volume={14},
   date={2009},
   pages={441--461},
   issn={1431-0635},
   review={\MR{2565906}},
}

\bib{Koshi_HSC_prime}{article}{
  doi = {10.48550/ARXIV.2202.03804},
  
  eprint = {https://arxiv.org/abs/2202.03804},
  
  author = {Koshikawa, Teruhisa},
  
  
  title = {The numerical Hodge standard conjecture for the square of a simple abelian variety of prime dimension},
  
  publisher = {arXiv},
  
  year = {2022},
  
}

\bib{MR172878}{article}{
   author={Lubin, Jonathan},
   author={Tate, John},
   title={Formal complex multiplication in local fields},
   journal={Ann. of Math. (2)},
   volume={81},
   date={1965},
   pages={380--387},
   issn={0003-486X},
   review={\MR{172878}},
   doi={10.2307/1970622},
}



\bib{MR1265538}{article}{
	author={Milne, J. S.},
	title={Motives over finite fields},
	conference={
		title={Motives},
		address={Seattle, WA},
		date={1991},
	},
	book={
		series={Proc. Sympos. Pure Math.},
		volume={55, Part 1},
		publisher={Amer. Math. Soc., Providence, RI},
	},
	isbn={0-8218-1636-5},
	date={1994},
	pages={401--459},
	review={\MR{1265538}},
	doi={10.1090/pspum/055.1/1265538},
}

\bib{MR1906596}{article}{
   author={Milne, J. S.},
   title={Polarizations and Grothendieck's standard conjectures},
   journal={Ann. of Math. (2)},
   volume={155},
   date={2002},
   number={2},
   pages={599--610},
   issn={0003-486X},
   review={\MR{1906596}},
   doi={10.2307/3062126},
}


\bib{MR249519}{article}{
   author={Milnor, John},
   title={On isometries of inner product spaces},
   journal={Invent. Math.},
   volume={8},
   date={1969},
   pages={83--97},
   issn={0020-9910},
   review={\MR{249519}},
   doi={10.1007/BF01404612},
}



\bib{MR1697859}{book}{
   author={Neukirch, J\"{u}rgen},
   title={Algebraic number theory},
   series={Grundlehren der mathematischen Wissenschaften [Fundamental
   Principles of Mathematical Sciences]},
   volume={322},
   note={Translated from the 1992 German original and with a note by Norbert
   Schappacher;
   With a foreword by G. Harder},
   publisher={Springer-Verlag, Berlin},
   date={1999},
   pages={xviii+571},
   isbn={3-540-65399-6},
   review={\MR{1697859}},
   doi={10.1007/978-3-662-03983-0},
}

\bib{MR2795752}{article}{
	author={O'Sullivan, Peter},
	title={Algebraic cycles on an abelian variety},
	journal={J. Reine Angew. Math.},
	volume={654},
	date={2011},
	pages={1--81},
	issn={0075-4102},
	review={\MR{2795752}},
	doi={10.1515/CRELLE.2011.025},
}

\bib{MR0338002}{book}{
   author={Saavedra Rivano, Neantro},
   title={Cat\'{e}gories Tannakiennes},
   language={French},
   series={Lecture Notes in Mathematics, Vol. 265},
   publisher={Springer-Verlag, Berlin-New York},
   date={1972},
   pages={ii+418},
   review={\MR{0338002}},
}


\bib{MR0354618}{book}{
   author={Serre, Jean-Pierre},
   title={Corps locaux},
   language={French},
   series={Publications de l'Universit\'{e} de Nancago, No. VIII},
   note={Deuxi\`eme \'{e}dition},
   publisher={Hermann, Paris},
   date={1968},
   pages={245},
   review={\MR{0354618}},
}



\bib{MR0498338}{book}{
   author={Serre, Jean-Pierre},
   title={Cours d'arithm\'{e}tique},
   language={French},
   series={Le Math\'{e}maticien, No. 2},
   note={Deuxi\`eme \'{e}dition revue et corrig\'{e}e},
   publisher={Presses Universitaires de France, Paris},
   date={1977},
   pages={188},
}



\bib{MR1484415}{book}{
	author={Serre, Jean-Pierre},
	title={Abelian $l$-adic representations and elliptic curves},
	series={Research Notes in Mathematics},
	volume={7},
	note={With the collaboration of Willem Kuyk and John Labute;
		Revised reprint of the 1968 original},
	publisher={A K Peters, Ltd., Wellesley, MA},
	date={1998},
	pages={199},
	isbn={1-56881-077-6},
	review={\MR{1484415}},
}


\bib{MR552586}{article}{
    AUTHOR = {Shioda, Tetsuji},
     TITLE = {The {H}odge conjecture for {F}ermat varieties},
   JOURNAL = {Math. Ann.},
    VOLUME = {245},
      YEAR = {1979},
    NUMBER = {2},
     PAGES = {175--184},
       DOI = {10.1007/BF01428804},
       
}

\bib{MR1922833}{article}{
   author={Tsuji, Takeshi},
   title={Semi-stable conjecture of Fontaine-Jannsen: a survey},
   note={Cohomologies $p$-adiques et applications arithm\'{e}tiques, II},
   journal={Ast\'{e}risque},
   number={279},
   date={2002},
   pages={323--370},
   issn={0303-1179},
   review={\MR{1922833}},
}


\end{biblist}
\end{bibdiv}
\end{document}